\documentclass[11pt,a4paper]{article}

\usepackage{epsf,epsfig,amsfonts,amsgen,amsmath,amstext,amsbsy,amsopn,amsthm,cases,listings,color
}
\usepackage{ebezier,eepic}
\usepackage{color}
\usepackage{multirow}
\usepackage{epstopdf}
\usepackage{graphicx}
\usepackage{pgf,tikz}
\usepackage{mathrsfs}
\usepackage[marginal]{footmisc}
\usepackage{enumitem}
\usepackage[titletoc]{appendix}
\usepackage{booktabs}
\usepackage{url}
\usepackage{amssymb}
\usepackage{subfigure}
\usepackage{authblk}
\usepackage{amssymb}

\usepackage{pgfplots}

\usepackage{subfigure}
\usepackage{mathrsfs}
\usepackage{xcolor}
\usetikzlibrary{arrows}

\allowdisplaybreaks[1]

\definecolor{uuuuuu}{rgb}{0.27,0.27,0.27}
\definecolor{sqsqsq}{rgb}{0.1255,0.1255,0.1255}

\newtheorem{definition}{Definition} [section]
\newtheorem{theorem}[definition]{Theorem}
\newtheorem{lemma}[definition]{Lemma}
\newtheorem{proposition}[definition]{Proposition}

\newtheorem{claim}[definition]{Claim}
\newtheorem{problem}[definition]{Problem}
\newtheorem{observation}[definition]{Observation}

\newenvironment{pf}{\noindent{\bf Proof}}

\newcommand{\hide}[1]{}
\def\multisets#1#2{\ensuremath{\left(\kern-.3em\left(\genfrac{}{}{0pt}{}{#1}{#2}\right)\kern-.3em\right)}}

\begin{document}
\title{Two stability theorems for $\mathcal{K}_{\ell + 1}^{r}$-saturated hypergraphs}

\date{\today}

\author[1]{Jianfeng Hou\thanks{Research was supported by National Natural Science Foundation of China (Grant No. 12071077). Email: jfhou@fzu.edu.cn}}
\author[1]{Heng Li\thanks{Email: hengli.fzu@gmail.com}}
\author[1]{Caihong Yang\thanks{Email: ych325123@outlook.com}}
\author[1]{Qinghou Zeng\thanks{Research was supported by National Natural Science Foundation of China (Grant No. 12001106) and National Natural Science Foundation of Fujian Province (Grant No. 2021J05128). Email: zengqh@fzu.edu.cn}}
\author[1]{Yixiao Zhang\thanks{Email: fzuzyx@gmail.com}}
\affil[1]{Center for Discrete Mathematics,
            Fuzhou University, Fujian, 350003, China}

\maketitle
\begin{abstract}
Let $\mathcal{F}$ be a family of  $r$-uniform hypergraphs (henceforth $r$-graphs). An $\mathcal{F}$-saturated $r$-graph is a maximal $r$-graph not containing any member of $\mathcal{F}$ as a subgraph. For integers $\ell \geq r \geq 2$, let $\mathcal{K}_{\ell + 1}^{r}$ be the collection of all $r$-graphs $F$ with at most $\binom{\ell+1}{2}$ edges such that for some $\left(\ell+1\right)$-set $S$ every pair $\{u, v\} \subset S$ is covered by an edge in $F$, and let $T_r (n, \ell)$ be the complete $\ell$-partite $r$-graph on $n$ vertices with no two part sizes differing by more than one. Let $t_{r}(n, \ell)$ be the number of edges in $T_r (n, \ell)$. Our first result shows that for each $\ell \geq r \geq 2$ every $\mathcal{K}_{\ell+1}^{r}$-saturated $r$-graph on $n$ vertices with $t_{r}(n, \ell) - o(n^{r-1+1/\ell})$ edges contains a complete $\ell$-partite subgraph on $(1-o(1))n$ vertices, which extends a stability theorem for $K_{\ell+1}$-saturated graphs given by Popielarz, Sahasrabudhe and Snyder. We also show that the bound is best possible.

Our second result is motivated by a celebrated theorem of Andr\'{a}sfai, Erd\H{o}s and S\'{o}s which states that for $\ell \geq 2$ every $K_{\ell+1}$-free graph $G$ on $n$ vertices with minimum degree $\delta(G) > \frac{3\ell-4}{3\ell-1}n$ is $\ell$-partite. We give a hypergraph version of it.   The \emph{minimum positive co-degree} of an $r$-graph $\mathcal{H}$, denoted by $\delta_{r-1}^{+}(\mathcal{H})$, is the maximum $k$ such that if $S$ is an $(r-1)$-set contained in a edge of $\mathcal{H}$, then $S$
is contained in at least $k$ distinct edges of $\mathcal{H}$. Let $\ell\ge 3$ be an integer and $\mathcal{H}$ be a $\mathcal{K}_{\ell+1}^3$-saturated $3$-graph on $n$ vertices. We prove that if either $\ell \ge 4$ and $\delta_{2}^{+}(\mathcal{H}) > \frac{3\ell-7}{3\ell-1}n$; or $\ell = 3$ and $\delta_{2}^{+}(\mathcal{H}) > 2n/7$, then $\mathcal{H}$ is $\ell$-partite; and the bound is best possible. This is the first stability result on minimum positive co-degree for hypergraphs.

\end{abstract}
\section{Introduction}\label{SEC:Introduction}
For an integer $r\ge 2$ an $r$-uniform hypergraph (henceforth $r$-graph) $\mathcal{H}$ is a collection of $r$-subsets of some finite set $V$. Given a family $\mathcal{F}$ of $r$-graphs we say $\mathcal{H}$ is $\mathcal{F}$-free if it does not contain any member of $\mathcal{F}$ as a subgraph. The {\em Tur\'{a}n number} $\mathrm{ex}(n,\mathcal{F})$ of $\mathcal{F}$ is the maximum number of edges in an $\mathcal{F}$-free $r$-graph on $n$ vertices. The {\em Tur\'{a}n density} $\pi(\mathcal{F} )$ of $\mathcal{F}$ is defined as
$\pi(\mathcal{F})=\lim_{n\to \infty}\mathrm{ex}(n,\mathcal{F})/{n\choose r}$. For $\ell \geq r \geq 2$, the generalized Tur\'{a}n graph, denote by $T_r (n, \ell)$, is the complete $\ell$-partite $r$-graph on $n$ vertices with no two part sizes differing by more than one. We use   $t_r(n, \ell)$ to denote the number of edges in $T_r(n, \ell)$.

The study of $\mathrm{ex}(n,\mathcal{F})$ is  the central topic in extremal combinaorics. One of the most famous results in this regard is Tur\'{a}n's theorem,
which states that for every integer $\ell \ge 2$ the Tur\'{a}n number $\mathrm{ex}(n,K_{\ell+1})$
is uniquely achieved by the Tur\'{a}n graph $T_2(n,\ell)$ (or $T(n,\ell)$ for simplicity). It was extended to general graphs by Erd\H{o}s, Stone and Simonovits~ \cite{ES66, ES46}. For $r\ge 3$ determining $\pi(\mathcal{F})$ for a family $\mathcal{F}$ of $r$-graphs
is known to be notoriously hard.
Let $\ell > r \ge 3$ be integers, and let $K_{\ell}^{r}$ be the \emph{complete $r$-graph} on $\ell$ vertices. The problem of determining $\pi(K_{\ell}^{r})$, raised by Tur\'{a}n~\cite{TU41},
 is still wide open.
Erd\H{o}s offered $\$ 500$ for the determination of any $\pi(K_{\ell}^{r})$
and $\$ 1000$ for the determination of all $\pi(K_{\ell}^{r})$.

A breakthrough on this topic was given by Mubayi \cite{mubayi2006hypergraph}. Note that the family $\mathcal{K}_{\ell+1}^r$ is the collection of all $r$-graphs $F$ with at most $\binom{\ell+1}{2}$ edges such that for some $(\ell+1)$-set $S$ every pair $\{x, y\}\subset S$ is covered by an edge in $F$. As a hypergraph extension  of the  classical Tur\'{a}n's theorem, Mubayi \cite{mubayi2006hypergraph} determined  the Tur\'{a}n number of  $\mathcal{K}_{\ell+1}^r$ by proving

\begin{theorem}[Mubayi \cite{mubayi2006hypergraph}]
Let $n, \ell, r \geq 2$ be integers. Then
\[
\mathrm{ex}(n, \mathcal{K}_{\ell+1}^r) = t_r(n, \ell),
\]
and the unique $r$-graph on $n$ vertices containing no copy of a member of $\mathcal{K}_{\ell+1}^r$ for which equality holds is $T_r(n, \ell)$.
\end{theorem}

Many families $\mathcal{F}$  have the property that there is a unique $\mathcal{F}$-free $r$-graph $\mathcal{G}$ on $n$ vertices with  $\mathrm{ex}(n,\mathcal{F})$ edges, and moreover, every $\mathcal{F}$-free $r$-graph $\mathcal{H}$ with the number of edges closed to $\mathrm{ex}(n,\mathcal{F})$ can be transformed to $\mathcal{G}$ by deleting and adding $o(n^r)$ edges.  Such a property is called the {\em stability} of $\mathcal{F}$. The Erd\H{o}s-Stone-Simonovits theorem \cite{ES66, ES46} and Erd\H{o}s-Simonovits stability theorem \cite{SI68} imply that every nondegenerate family of graphs is stable. The stability  of $\mathcal{K}_{\ell+1}^r$ was given by Mubayi \cite{mubayi2006hypergraph}. In this paper, we are interested in the stability results of maximal $\mathcal{K}_{\ell+1}^r$-free $r$-graphs. An $r$-graph $\mathcal{H}$ is called
$\mathcal{F}$-\emph{saturated} if $\mathcal{H}$ is $\mathcal{F}$-free and whenever a new edge $E$  is added to $\mathcal{H}$,  then $\mathcal{H}\cup \{E\}$ contains a copy of   a member of $\mathcal{F}$. The topic on the stability of $\mathcal{F}$-saturated graphs was widely studied when $r=2$, see \cite{gerbner2021note, wang2022counterexamples,wang2022stability}. For complete graphs, Tyomkyn and Uzzell \cite{tyomkyn2015strong} proved that every $K_{4}$-saturated with $t_2(n,3) - cn$ edges contains a complete $3$-partite subgraph on $(1 - o(1))n$ vertices. In the same paper, they asked  when can one guarantee an ``almost spanning'' complete $\ell$-partite subgraph in a  $K_{\ell+1}$-saturated graph  with $t_2(n, \ell)-o(n^2)$ edges.  Popielarz, Sahasrabudhe and Snyder \cite{popielarz2018stability}  answered the question by proving the following theorem and gave a tight bound through an ingenious construction.


\begin{theorem}[Popielarz, Sahasrabudhe and Snyder \cite{popielarz2018stability}]\label{th8}
Let $\ell \geq 2$ be an integer. Every $K_{\ell+1}$-saturated graph $G$ on $n$ vertices with $t_{2}(n,\ell) - o(n^{1+1/\ell})$ edges contains a complete $\ell$-partite subgraph on $(1-o(1))n$ vertices.
\end{theorem}
Our first result extends Theorem \ref{th8} to  $\mathcal{K}_{\ell+1}^r$-saturated $r$-graphs by proving
\begin{theorem}\label{th2}
Let ${\ell} \geq r \geq 3$ be integers. Every $\mathcal{K}_{\ell+1}^{r}$-saturated $r$-graph $\mathcal{H}$ on $n$ vertices with $  t_{r}(n, \ell) - o(n^{r-1+1/\ell})$ edges contains a complete $\ell$-partite subgraph on $(1 - o(1))n$ vertices.
\end{theorem}

We also prove that the bound in  Theorem \ref{th2} is tight in the sense that for  some real  $\sigma > 0$ there exists an $r$-graph $\mathcal{H}$ with $t_r(n,\ell) - \sigma n^{r-1+1/\ell}$ edges such that the conclusion of Theorem \ref{th2} fails.

Another stability form is under  the degree condition, and the celebrated theorem on this topic was given by Andr\'{a}sfai, Erd\H{o}s and S\'{o}s \cite{andrasfai1974connection} as follows.

\begin{theorem}[Andr\'{a}sfai, Erd\H{o}s and S\'{o}s \cite{andrasfai1974connection}] \label{AES}
Every $K_{\ell+1}$-free graph $G$ on $n$ vertices with minimum degree $\delta(G)>\frac{3\ell-4}{3\ell-1}n$ is $\ell$-partite.
\end{theorem}

Note that the bound on minimum degree in Theorem \ref{AES} is tight and the extremal graph is far from the Tur\'{a}n graph $T(n,\ell+1)$ in edit-distance. Motivated by the degree versions of the Erd\H{o}s-Ko-Rado theorem and co-degree
Tur\'an numbers,  Balogh, Lemons and Palmer~\cite{Balogh2021Siam} posed the positive co-degree as a reasonable notion of ``minimum degree'' in a hypergraph.

\begin{definition}[Minimum positive co-degree]\label{Def:positive}
The minimum positive co-degree of a non-empty $r$-graph $\mathcal{H}$, denoted by $\delta_{r-1}^{+}(\mathcal{H})$, is the maximum $k$  such that if $S$ is an $(r-1)$-set contained in a edge of $\mathcal{H}$, then $S$
is contained in at least $k$ distinct edges of $\mathcal{H}$.
\end{definition}

Our second result  gives a 3-graph version of  Andr\'{a}sfai-Erd\H{o}s-S\'{o}s Theorem.

\begin{theorem}\label{main result}
Let $\ell \geq 3$ be an integer, and define
\[
f(\ell)= \begin{cases}
2/7 \quad & \ell=3,\\
\frac{3\ell-7}{3\ell-1} \quad & \ell \geq 4.
\end{cases}
\]
Then every $\mathcal{K}_{\ell+1}^3$-saturated $3$-graph $\mathcal{H}$ on $n$ vertices with $\delta_{2}^{+}(\mathcal{H}) > f(\ell) n$ is $\ell$-partite.
\end{theorem}

We remark that this is  the first stability result on minimum positive co-degree for hypergraphs. We also show that this bound in Theorem \ref{main result} is best possible, and the extremal construction is far from  $T_3(n,\ell+1)$ in edit-distance.

The remainder of this paper is organized as follows. In Section~\ref{SEC:Prelim}, we introduce some preliminary definitions and results. In Section~\ref{SEC:3}, we prove Theorem~\ref{th2} and show that the bound is tight using a construction given by Popielarz, Sahasrabudhe and Snyder \cite{popielarz2018stability}. In Section~\ref{SEC:4}, we prove Theorem~\ref{main result} and show its tightness. The final section contains some concluding remarks.

\section{Preliminaries}\label{SEC:Prelim}
In this section, we  introduce some preliminary definitions and results that will be used later. Let $r\ge 2$ be an  integer and $\mathcal{H}$ be an $r$-graph. We use $v(\mathcal{H})$ and $|\mathcal{H}|$ (or $e(\mathcal{H})$ if $r=2$) to denote the number of vertices and edges in $\mathcal{H}$, respectively. Let $G$ be a graph. For $v \in V(G)$, we use $N_G(v)$ and $d_G(v)$ to denote the \emph{neighborhood} and \emph{degree} of $v$ in $G$, respectively.  We use $\Delta(G)$ (or $\delta(G)$) denote the maximum (or minimum) degree of $G$.  The \emph{complement} of $G$, denoted by $\overline{G}$, is a graph on the same set of vertices of $G$ such that there is  an edge between two vertices $u,v$ in $\overline{G}$, if and only if there is no edge in between $u,v$  in $G$. For subsets $S, T \subseteq V(G)$, we denote by $G[S]$ the induced subgraph of $G$ with $S$, and $G[S, T]$ the induced bipartite subgraph of $G$ with two parts $S$ and $T$. Let $N_G(S) = \cap_{v\in S}N_G(v)$ be the common neighbours of $S$ in $G$, and $I_G(S)$ be  the collection of edges of $G$ covered by $S$. For an edge $e \in E(G)$,  let $t_G(e)$ be the number of triangles containing $e$, and let $t^{+}(G)$ be the maximum $k$ such that  if an edge $e$ is contained in a triangle of $G$, then $e$ is contained in at least $k$ distinct triangles of $G$. We will  omit the subscript when it is clear to which graph we refer.

If $G$ is an $\ell$-partite graph with the vertex partition $V_1,\dots,V_\ell$, then we use $\widehat{G}$ to denote the $\ell$-\emph{partite complement} containing all the non-edges of $G$ between vertex class $V_i$ and $V_j$ for $1\leq i < j\leq \ell$ of $G$. An \emph{$(\ell+1)$-saturating edge} in $G$ is an edge of $\overline{G}$  whose addition to $G$  forms at least one copy of $K_{\ell+1}$ in $G$.

For $\ell \geq 2$ and $ 0 \leq k \leq \ell -2$, a \emph{$5$-wheel-like} graph $W_{\ell,k}$ is a graph consisting of two cliques $Q_1$, $Q_2$ of order $\ell - 1$, which intersect in exactly $k$ vertices, together with a vertex $v$, adjacent to all vertices of $Q_1$ and $Q_2$ and an edge $u_1u_2$, where $u_1$ is adjacent to the vertices of $Q_1$ and $u_2$ is adjacent to the vertices of $Q_2$ (see Figure 1). For example, $W_{2,0}$ is the $5$-cycle, and $W_{3,1}$ is the wheel with $5$ spokes. We call the vertex $v$ the \emph{top} and the edge $u_1u_2$ the \emph{bottom} of $W_{\ell,k}$.

\begin{figure}[htbp]
\centering
\tikzset{every picture/.style={line width=0.75pt}} 

\begin{tikzpicture}[x=0.75pt,y=0.75pt,yscale=-1,xscale=1]

\draw  [fill={rgb, 255:red, 0; green, 0; blue, 0 }  ,fill opacity=1 ][line width=0.75]  (370.8,225.9) .. controls (370.8,223.64) and (372.64,221.8) .. (374.9,221.8) .. controls (377.16,221.8) and (379,223.64) .. (379,225.9) .. controls (379,228.16) and (377.16,230) .. (374.9,230) .. controls (372.64,230) and (370.8,228.16) .. (370.8,225.9) -- cycle ;
\draw  [fill={rgb, 255:red, 0; green, 0; blue, 0 }  ,fill opacity=1 ][line width=0.75]  (420.8,125.9) .. controls (420.8,123.64) and (422.64,121.8) .. (424.9,121.8) .. controls (427.16,121.8) and (429,123.64) .. (429,125.9) .. controls (429,128.16) and (427.16,130) .. (424.9,130) .. controls (422.64,130) and (420.8,128.16) .. (420.8,125.9) -- cycle ;
\draw  [fill={rgb, 255:red, 0; green, 0; blue, 0 }  ,fill opacity=1 ][line width=0.75]  (221.8,124.9) .. controls (221.8,122.64) and (223.64,120.8) .. (225.9,120.8) .. controls (228.16,120.8) and (230,122.64) .. (230,124.9) .. controls (230,127.16) and (228.16,129) .. (225.9,129) .. controls (223.64,129) and (221.8,127.16) .. (221.8,124.9) -- cycle ;
\draw  [fill={rgb, 255:red, 0; green, 0; blue, 0 }  ,fill opacity=1 ][line width=0.75]  (321.8,36.9) .. controls (321.8,34.64) and (323.64,32.8) .. (325.9,32.8) .. controls (328.16,32.8) and (330,34.64) .. (330,36.9) .. controls (330,39.16) and (328.16,41) .. (325.9,41) .. controls (323.64,41) and (321.8,39.16) .. (321.8,36.9) -- cycle ;
\draw  [fill={rgb, 255:red, 0; green, 0; blue, 0 }  ,fill opacity=1 ][line width=0.75]  (271.8,225.9) .. controls (271.8,223.64) and (273.64,221.8) .. (275.9,221.8) .. controls (278.16,221.8) and (280,223.64) .. (280,225.9) .. controls (280,228.16) and (278.16,230) .. (275.9,230) .. controls (273.64,230) and (271.8,228.16) .. (271.8,225.9) -- cycle ;
\draw    (325.9,36.9) -- (424.9,124.3) ;
\draw   (226.9,125.5) .. controls (226.9,116.39) and (256.87,109) .. (293.85,109) .. controls (330.83,109) and (360.8,116.39) .. (360.8,125.5) .. controls (360.8,134.61) and (330.83,142) .. (293.85,142) .. controls (256.87,142) and (226.9,134.61) .. (226.9,125.5) -- cycle ;
\draw   (290.8,125.5) .. controls (290.8,116.39) and (320.35,109) .. (356.8,109) .. controls (393.25,109) and (422.8,116.39) .. (422.8,125.5) .. controls (422.8,134.61) and (393.25,142) .. (356.8,142) .. controls (320.35,142) and (290.8,134.61) .. (290.8,125.5) -- cycle ;
\draw  [fill={rgb, 255:red, 0; green, 0; blue, 0 }  ,fill opacity=1 ][line width=0.75]  (247.8,125) .. controls (247.8,123.34) and (249.14,122) .. (250.8,122) .. controls (252.46,122) and (253.8,123.34) .. (253.8,125) .. controls (253.8,126.66) and (252.46,128) .. (250.8,128) .. controls (249.14,128) and (247.8,126.66) .. (247.8,125) -- cycle ;
\draw  [fill={rgb, 255:red, 0; green, 0; blue, 0 }  ,fill opacity=1 ][line width=0.75]  (272.8,125) .. controls (272.8,123.34) and (274.14,122) .. (275.8,122) .. controls (277.46,122) and (278.8,123.34) .. (278.8,125) .. controls (278.8,126.66) and (277.46,128) .. (275.8,128) .. controls (274.14,128) and (272.8,126.66) .. (272.8,125) -- cycle ;
\draw  [fill={rgb, 255:red, 0; green, 0; blue, 0 }  ,fill opacity=1 ][line width=0.75]  (309.8,125) .. controls (309.8,123.34) and (311.14,122) .. (312.8,122) .. controls (314.46,122) and (315.8,123.34) .. (315.8,125) .. controls (315.8,126.66) and (314.46,128) .. (312.8,128) .. controls (311.14,128) and (309.8,126.66) .. (309.8,125) -- cycle ;
\draw  [fill={rgb, 255:red, 0; green, 0; blue, 0 }  ,fill opacity=1 ][line width=0.75]  (334.8,125) .. controls (334.8,123.34) and (336.14,122) .. (337.8,122) .. controls (339.46,122) and (340.8,123.34) .. (340.8,125) .. controls (340.8,126.66) and (339.46,128) .. (337.8,128) .. controls (336.14,128) and (334.8,126.66) .. (334.8,125) -- cycle ;
\draw  [fill={rgb, 255:red, 0; green, 0; blue, 0 }  ,fill opacity=1 ][line width=0.75]  (396.8,125) .. controls (396.8,123.34) and (398.14,122) .. (399.8,122) .. controls (401.46,122) and (402.8,123.34) .. (402.8,125) .. controls (402.8,126.66) and (401.46,128) .. (399.8,128) .. controls (398.14,128) and (396.8,126.66) .. (396.8,125) -- cycle ;
\draw  [fill={rgb, 255:red, 0; green, 0; blue, 0 }  ,fill opacity=1 ][line width=0.75]  (371.8,125) .. controls (371.8,123.34) and (373.14,122) .. (374.8,122) .. controls (376.46,122) and (377.8,123.34) .. (377.8,125) .. controls (377.8,126.66) and (376.46,128) .. (374.8,128) .. controls (373.14,128) and (371.8,126.66) .. (371.8,125) -- cycle ;
\draw    (225.9,123) -- (325.9,36.9) ;
\draw    (225.9,124.9) -- (275.9,225.9) ;
\draw    (374.9,225.9) -- (275.9,225.9) ;
\draw    (374.9,225.9) -- (424.9,125.9) ;
\draw    (325.9,36.9) -- (290.8,125.5) ;
\draw    (275.9,225.9) -- (290.8,125.5) ;
\draw    (275.9,225.9) -- (360.8,127.5) ;
\draw    (374.9,225.9) -- (290.8,127.5) ;
\draw    (325.9,36.9) -- (360.8,125.5) ;
\draw    (374.9,225.9) -- (360.8,125.5) ;

\draw (267,233) node [anchor=north west][inner sep=0.75pt]    {$u_{1}$};
\draw (366,233) node [anchor=north west][inner sep=0.75pt]    {$u_{2}$};
\draw (320,19.4) node [anchor=north west][inner sep=0.75pt]    {$v$};
\draw (198,117.4) node [anchor=north west][inner sep=0.75pt]    {$Q_{1}$};
\draw (431,117.4) node [anchor=north west][inner sep=0.75pt]    {$Q_{2}$};

\end{tikzpicture}
\caption{A $5$-wheel-like graph.}
\end{figure}
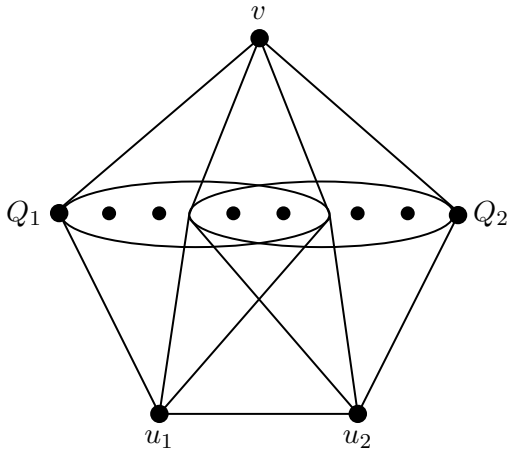

Brandt~\cite{B03com} generalized the observation that every maximal triangle-free graph is either complete bipartite or contains a $5$-cycle and proved the following result.

\begin{lemma}[Brandt~\cite{B03com}]\label{Lemma:maximal}
Let $G$ be a $K_{\ell+1}$-saturated graph. Then either $G$ is complete $\ell$-partite, or $G$ contains a $5$-wheel-like subgraph.
\end{lemma}

For  graphs $G$ and $H$, let $N(H, G)$ denote the number of copies of $H$ in $G$. The main idea in the proof of Theorem~\ref{th2} is to consider $K_{\ell+1}$-free graphs that maximizing $K_r$. The following definition plays a key role in our proof.

\begin{definition}[$K_r$-maximal, $K_{\ell+1}$-free graph]\label{Kr maximal}Let $\ell\ge r\ge 2$ be integers.
A graph $G$ is $K_r$-maximal, $K_{\ell+1}$-free if $G$ is $K_{\ell+1}$-free but $G+E'$ contains a copy of $K_{\ell+1}$ for each $E'\subseteq E(\overline{G})$ with   $N(K_r, G+E')>N(K_r, G)$.
\end{definition}

The following lemma, given by Popielarz, Sahasrabudhe and Snyder \cite{popielarz2018stability},   will help us to change an $\ell$-partite graph to a complete $\ell$-partite graph by deleting small number of vertices.
\begin{lemma}[Popielarz, Sahasrabudhe and Snyder \cite{popielarz2018stability}]\label{lem}
Let $\ell, t \geq 1$ be an integer with $\ell - t \geq 1$ and let $G$ be a $K_{\ell + 1}$-free, $(\ell+1)$-partite graph with vertex classes $A, B, X_{1},\ldots,X_{\ell - 1}$ such that  the following conditions hold:
\begin{itemize}
  \item [1.] $E \subseteq E_{\widehat{G}}(A, B)$ is a collection of non-edges between $A, B$.
  \item [2.] There exist $K_{\ell + 1 - t}$-free subgraphs $H_{1},\ldots,H_{s} \subseteq G$ such that every edge of $E$ is $(\ell + 1 - t)$-saturating in at least one of the graphs $H_{1},\ldots,H_{s}$.
\end{itemize}
Then, there exists a set $R' \subseteq A \cup B$ covering every edge of $E$ with
$$|I_{\widehat{G}}(R')| \geq c_{\ell,t}'s^{-\frac{1}{\ell - t}}|R'|^{\frac{\ell + 1 - t}{\ell - t}}$$
 where $c'_{\ell,t}$ is a constant depending only on $\ell,t$.
\end{lemma}

Now we return to $r$-graphs. Let $r \ge 2$ be an  integer and $\mathcal{H}$ be an $r$-graph. For a vertex $u \in V(\mathcal{H})$,  we use $d_{\mathcal{H}}(u)$ to denote the \emph{degree} of $u$, and $\delta(\mathcal{H})$ to denote the \emph{minimum degree} of $\mathcal{H}$.
The \emph{link} of $u$ is defined as
$$L_{\mathcal{H}}(u) = \{A\subset V(\mathcal{H})\colon A\cup\{u\}\in\mathcal{H}\}.$$
Clearly, $L_{\mathcal{H}}(u)$ is an $(r-1)$-graph with $d_{\mathcal{H}}(u)$ edges. For  $1 \leq i \leq r-1$ the {\em $i$-th shadow} of $\mathcal{H}$ is the $(r-i)$-graph defined by
\[
\partial_i\mathcal{H} = \left\{A \in \binom{V\left(\mathcal{H}\right)}{r-i} : \exists B \in \mathcal{H}\,\, {\rm such}\,\, {\rm that}\,\, A \subset B\right\}.
\]
We use $\partial\mathcal{H}=\partial_1\mathcal{H}$ to denote the \emph{shadow} of $\mathcal{H}$. Note that $\partial_{r-2}\mathcal{H}$ is a graph, and an edge in $\mathcal{H}$  creates a copy of  $K_r$ in $\partial_{r-2}\mathcal{H}$. The following easy observations give relationships between  $\mathcal{H}$ and $\partial_{r-2}\mathcal{H}$.
\begin{observation}[Liu \cite{liu2021new}] \label{ob1}
Let $\mathcal{H}$  be an $r$-graph. Then the following holds.
\begin{itemize}
  \item [(a)] $\mathcal{H}$ is $\mathcal{K}_{\ell + 1}^{r}$-free iff $\partial_{r - 2}\mathcal{H}$ is $K_{\ell + 1}$-free.
  \item [(b)] The number of edges in $\mathcal{H}$ is at most the number of copies of $K_{r}$ in $\partial_{r - 2}\mathcal{H}$.
\end{itemize}
\end{observation}

\begin{observation}\label{l-partite}
Let ${\ell} \geq r \geq 2$ be integers, and $\mathcal{H}$  be an $r$-graph. Then  $\mathcal{H}$ is $\ell$-partite if and only if $\partial_{r - 2}\mathcal{H}$ is $\ell$-partite.
\end{observation}
\begin{proof}
If $\mathcal{H}$ is $\ell$-partite with the vertex partition $V(\mathcal{H}) = V_1 \cup \dots \cup V_\ell$, then $|e\cap V_i| \leq 1$ for each $e\in \mathcal{H}$ and $i\in [\ell]$, which implies $\partial_{r - 2}\mathcal{H}$ is $\ell$-partite with the vertex partition $V(\mathcal{H}) = V_1 \cup \dots \cup V_\ell$. For the other side, assume that $\partial_{r - 2}\mathcal{H}$ is $\ell$-partite with the vertex partition $V(\mathcal{H}) = V_1 \cup \dots \cup V_\ell$. If $\mathcal{H}$ is not $\ell$-partite, then for any partition $V(\mathcal{H}) = W_1 \cup \dots \cup W_\ell$, there is a edge $e\in \mathcal{H}$ and $i\in [\ell]$ such that $|e\cap W_i|\ge 2$. Especially, the statement  holds for $V_1,\dots,V_\ell$, which implies $\partial_{r - 2}\mathcal{H}[V_i]$ is not empty for some $i\in [\ell]$, a contradiction.
\end{proof}

The following lemma gives a relationship between  $\mathcal{K}_{\ell+1}^{r}$-saturated graphs and its  $(r-2)$-th shadows.
\begin{lemma}\label{Maximal Kl if and only if maximal Kr}
Let $\ell\ge r\ge 3$ be  integers and $\mathcal{H}$ be an $r$-graph. Then $\mathcal{H}$ is $\mathcal{K}_{\ell+1}^{r}$-saturated if and only if each copy of $K_r$ in $\partial_{r-2} \mathcal{H}$ forms an edge in $\mathcal{H}$, and  $\partial_{r-2} \mathcal{H}$ is $K_r$-maximal, $K_{\ell + 1}$-free.
\end{lemma}
\begin{proof}
On the one hand, suppose that $\mathcal{H}$ is a $\mathcal{K}_{\ell+1}^{r}$-saturated $r$-graph. Then  $\partial_{r-2} \mathcal{H}$ is $K_{\ell+1}$-free by Observation \ref{ob1}. Now we show  $\partial_{r-2} \mathcal{H}$ is  $K_r$-maximal. Otherwise, there is a subset $E\subseteq E(\overline{\partial_{r-2} \mathcal{H}})$ such that $\partial_{r-2} \mathcal{H}+E$ is also $K_{\ell+1}$-free but $N(K_r,\partial_{r-2} \mathcal{H}+E)> N(K_r,\partial_{r-2} \mathcal{H})$. We may choose $E$ such that $|E|$ is minimal and assume that $U=\{u_1,u_2,\dots,u_r\}$ is the subset such that $G=(\partial_{r-2} \mathcal{H}+E)[U]$ is a copy of $K_r$. Then $E\subseteq E(G)$, which means $\{u_1,u_2,\dots,u_r\}\notin \mathcal{H}$. Let $\mathcal{H}'$ = $\mathcal{H} \cup \left\{ u_1,u_2,\dots,u_r \right\}$. Then $\partial_{r-2} \mathcal{H}'=\partial_{r-2} \mathcal{H}+E$ and so $\partial_{r-2} \mathcal{H}'$ is $K_{\ell+1}$-free. However, since $\mathcal{H}$ is $\mathcal{K}_{\ell+1}^{r}$-saturated, $\mathcal{H}'$ contains a copy of a member of $\mathcal{K}_{\ell + 1}^{r}$, which implies that $\partial_{r-2} \mathcal{H}'$ contains a copy of $K_{\ell+1}$ by Observation \ref{ob1}, a contradiction.

Suppose that there exists an $r$-set $\{u_1,u_2,\dots,u_r\}$ which forms a $K_r$ in $\partial_{r-2} \mathcal{H}$ but not an edge in $\mathcal{H}$.  Let $\mathcal{H}'$ = $\mathcal{H} \cup \left\{ u_1,u_2,\dots,u_r \right\}$.  Then $\partial_{r-2} \mathcal{H} = \partial_{r-2} \mathcal{H}'$ and  so $\partial_{r-2} \mathcal{H}'$ is $K_{\ell+1}$-free. However, it follows from $\mathcal{H}$ is $\mathcal{K}_{\ell+1}^{r}$-saturated that $\mathcal{H}'$ contains a copy of a member of $\mathcal{K}_{\ell + 1}^{r}$-free, and then $\partial_{r-2} \mathcal{H}'$ contains a copy of $K_{\ell+1}$ by Observation \ref{ob1}, a contradiction.

On the other hand, suppose each $K_r$ in $\partial_{r-2} \mathcal{H}$ forms an edge in $\mathcal{H}$, and  $\partial_{r-2} \mathcal{H}$ is $K_r$-maximal, $K_{\ell + 1}$-free. By Observation \ref{ob1}, $\mathcal{H}$ is $\mathcal{K}_{\ell+1}^{r}$-free. If there exists $r$-set $\{u_1,\dots,u_r\}$ not in $\mathcal{H}$  such that $\mathcal{H}'=\mathcal{H}\cup \{u_1,\dots,u_r\}$ is also $\mathcal{K}_{\ell+1}^{r}$-free.  Again, by Observation \ref{ob1}, $\partial_{r-2} \mathcal{H}'$ is $K_{\ell+1}$-free. Since every $K_r$ in $\partial_{r-2} \mathcal{H}$ is a edge in $\mathcal{H}$, the $r$-set $\{u_1,\dots,u_r\}$ fails to form a $K_r$ in $\partial_{r-2} \mathcal{H}$. Thus, $N(K_r, \partial_{r-2} \mathcal{H}')>N(K_r, \partial_{r-2} \mathcal{H})$. It follows from $\partial_{r-2} \mathcal{H}$ is $K_r$-maximal that $\partial_{r-2} \mathcal{H}$ contains a copy of $K_{\ell+1}$, and so does $\partial_{r-2} \mathcal{H}'$, a contradiction.
\end{proof}

We end this section by a result of Fisher and Ryan, which gives a relationship between the number of large cliques and   small ones in a $K_{\ell+1}$-free graph.
\begin{theorem}[Fisher and Ryan \cite{fisher1992bounds}]\label{th1}
Let $G$ be a  $K_{\ell + 1}$-free graph on $n$ vertices. For every $i \in [\ell]$, let $k_i$ denote the number of copies of $K_i$ in $G$. Then
\[
\left(\frac{k_{\ell}}{\binom{\ell}{\ell}} \right)^{\frac{1}{\ell}} \leq \left(\frac{k_{\ell - 1}}{\binom{\ell}{\ell - 1}}\right)^{\frac{1}{\ell - 1}} \leq\cdots\leq \left(\frac{k_{2}}{\binom{\ell}{2}} \right)^{\frac{1}{2}} \leq \left(\frac{k_{1}}{\binom{\ell}{1}} \right)^{\frac{1}{1}}.
\]
\end{theorem}

\section{Large induced complete $\ell$-partite subgraph}\label{SEC:3}
\subsection{Proof of Theorem \ref{th2}}
In this subsection, we prove the following quantitative version of Theorem \ref{th2} using an idea given by Popielarz, Sahasrabudhe and Snyder \cite{popielarz2018stability}.

\begin{theorem}\label{final result}
Let $\ell$, $r$ and $n$ be integers with ${\ell} \geq r \geq 3$ and  $n \geq 10^{20}\ell^{20}$,  and let $\varepsilon$ be a real number with $\varepsilon \leq n^{1/\ell-1}$. Suppose that $\mathcal{H}$ is  a $\mathcal{K}_{\ell+1}^{r}$-saturated $r$-graph on $n$ vertices with
\[
|\mathcal{H}|\ge t_r(n, \ell) - \frac{\binom{\ell}{r}r\varepsilon}{4(\ell-1)\ell^{r-1}} n^{r}.
\]
Then there exists a constant $C_{\ell}>0$ such that $\mathcal{H}$ contains a complete $\ell$-partite $r$-graph on $\left(1 - C_{\ell}\varepsilon n^{1- 1/\ell} - 10^6 \ell^3(3r-4)^3 \varepsilon \right)n$ vertices.
\end{theorem}
\begin{proof}

Recall that $t_r(n, \ell)$  denotes the number of edges in the Tur\'{a}n graph $T_r(n, \ell)$. First, we give a lower bound on $t_r(n, \ell)$. Let $n = \ell q+s$, where $0 \leq s \leq \ell - 1$. Then
\begin{align}\label{number-edge-t-r}
\binom{\ell}{r}\left(\frac{n}{\ell}\right)^r&= \binom{\ell}{r}\left(q + \frac{s}{\ell}\right)^r   \notag \\
&=\binom{\ell}{r}q^r + s\binom{\ell-1}{r-1}q^{r-1}+\sum_{i=2}^r\binom{\ell}{r}\binom{r}{i}\left(\frac{s}{\ell}\right)^iq^{r-i}
\end{align}
Note that the Tur\'{a}n graph $T_r(n, \ell)$ can be considered as the Tur\'{a}n graph $T_r(\ell q, \ell)$ and $s$ additional vertices such that the link of each vertex is the Tur\'{a}n graph $T_{r-1}((\ell-1) q+s-1, \ell-1)$. So,
\begin{align}\label{edge-t-r}
t_r(n, \ell) \ge t_r(\ell q, \ell) + \sum_{i = 1}^{r}\binom{s}{i}\binom{\ell-i}{r-i}q^{r-i}
\ge \binom{\ell}{r}q^r + s\binom{\ell-1}{r-1}q^{r-1}.
\end{align}
Combining \eqref{number-edge-t-r} and \eqref{edge-t-r}, we have
\begin{align}\label{lower-bound-turan-graph}
t_r(n, \ell)-\binom{\ell}{r}\left(\frac{n}{\ell}\right)^r \geq  - \sum_{i=2}^r\binom{\ell}{r}\binom{r}{i}\left(\frac{s}{\ell}\right)^iq^{r-i} \geq -\ell^2r^{r+1}n^{r-2}.
\end{align}

Suppose that $\mathcal{H}$ is a  $\mathcal{K}_{\ell+1}^{r}$-saturated $r$-graph on $n$ vertices with
\[
\left |\mathcal{H}\right | \geq t_{r}(n, \ell) - \frac{\binom{\ell}{r}r\varepsilon}{2(\ell-1)\ell^{r-1}} n^{r}.
 \]
We may assume that
\begin{align}\label{bound-varepsilon}
\varepsilon\ge \frac{2r^r(\ell-1)\ell^{r+1}}{\binom{\ell}{r}}n^{-2}.
\end{align}
Then by Lemma $\ref{Maximal Kl if and only if maximal Kr}$, $G^\ast =\partial_{r-2}\mathcal{H}$ is $K_r$-maximal, $K_{\ell+1}$-free. By Observation \ref{ob1}(b), and inequalities \eqref{lower-bound-turan-graph} and \eqref{bound-varepsilon}, we obtain
\begin{align}\label{number-edge-r-2-H}
N(K_r, G^\ast)&\ge t_r(n,\ell) - \frac{\binom{\ell}{r}r\varepsilon}{2(\ell-1)\ell^{r-1}} n^{r} \notag  \\
&\ge \binom{\ell}{r}\left(\frac{n}{\ell}\right)^r - \frac{\binom{\ell}{r}r\varepsilon}{(\ell-1)\ell^{r-1}}  n^{r}.
\end{align}

It suffices to find a complete $\ell$-partite subgraph of $G^\ast$ on $(1-o(1))n$ vertices by Observation \ref{l-partite}. Now we bound the number of edges of $G^\ast$. By Theorem $\ref{th1}$ and \eqref{number-edge-r-2-H}, we obtain
\begin{align}\label{lower-bound-r-2-H}
e(G^\ast) \geq & \left(\frac{N(K_r, G^\ast)}{\binom{\ell}{r}}\right)^{2/r} \binom{\ell}{2} \notag
\\ \geq &  \left(\frac{\binom{\ell}{r}\left(\frac{n}{\ell}\right)^r - \frac{\binom{\ell}{r}r}{(\ell-1)\ell^{r-1}} \varepsilon n^{r}}{\binom{\ell}{r}}\right)^{2/r} \binom{\ell}{2}    \notag
\\ \geq & \binom{\ell}{2}\left(\frac{n}{\ell}\right)^2\left(1 - \frac{\ell r\varepsilon}{(\ell-1)}\right)^{2/r}  \notag
\\ \geq & \binom{\ell}{2}\left(\frac{n}{\ell}\right)^2\left(1- \frac{2\ell \varepsilon}{(\ell-1)}\right)   \notag
\\ \geq & t_2(n, \ell) - \varepsilon n^{2},
\end{align}
where the second inequality from the bottom  follows from Bernoulli inequality $(1 + x)^{a} \geq 1 + ax$ for $x \geq -1$ and $a >0$.

In the following, let
\[
\eta = \frac{1}{10^5\ell^3(3r-4)^3}.
\]
Then
\begin{align}\label{l-1/l-delta}
\frac{\ell-1}{\ell}-\eta\ge \frac{3\ell-4}{3\ell-1}.
\end{align}

For a graph $G$ and $v\in V(G)$, $v$ is called a \emph{small vertex} of $G$ if $d(v)<(\frac{\ell-1}{\ell}-\eta)v(G)$. We get an $\ell$-partite subgraph $G'$ of $G^\ast$ as follows:
If $G^\ast$ has no small vertex, then $G'=G^\ast$ by Theorem \ref{AES} and \eqref{l-1/l-delta}. Otherwise, we  delete a small vertex of $G^\ast$ and get a new graph. For this new graph, we do the same operation until the remaining graph $G'$  has no small vertices. Then $G'$ is $\ell$-partite by Theorem \ref{AES}. Let $T=V(G^\ast)-V(G')$ be the deleted vertices. For simplicity, let $t=|T|$ and $n'=n-t$. We bound $|T|$ as follows:

\begin{claim}\label{T}
$|T| \leq 10^6 \ell^3(3r-4)^3 \varepsilon n$.
\end{claim}
\begin{proof}[Proof of Claim \ref{T}]
It follows from $G'$ is $K_{\ell + 1}$-free that $e(G') \leq t_2(n', \ell)$.
If $|T| > 10^6 \ell^3(3r-4)^3 \varepsilon n$, then
\begin{align}
e(G^\ast)&<\sum_{i=n'+1}^n\left(\frac{\ell - 1}{\ell}-\eta\right)i + t_2(n', \ell) \notag\\
&= \sum_{i=n'+1}^n\left(\frac{\ell - 1}{\ell}i-1\right) - \sum_{i=n'+1}^n(\eta i-1) + t_2(n', \ell)  \notag \\
& < t_2(n, \ell) - \sum_{i=n'+1}^n(\eta i-1)  \notag \\
&\leq t_2(n, \ell)-\frac{\eta t}{2}\left(2n+1-t\right)  \notag \\
& < t_2(n, \ell) - \varepsilon n^{2}, \notag
\end{align}
a contradiction. Note that the last inequality holds as $\frac{\eta t}{2}\left(2n+1-t\right)$ is increasing for $t$ on the interval $\left(10^6 \ell^3(3r-4)^3 \varepsilon n, n+1/2\right)$.
\end{proof}

Suppose that $G'$ is an  $\ell$-partite graph with the vertex partition $V_1 \cup \dots \cup V_\ell$ such that
\begin{align}\label{min-deg-G'}
\delta(G') \geq \left(\frac{\ell-1}{\ell}-\eta\right)n'.
\end{align}
Then the following claim is obvious.

\begin{claim}\label{claim Vi}
The following statements hold.
\begin{itemize}
\item [(a)]$\left||V_i|- \frac{1}{\ell}n'\right| \leq \eta^{1/2}n'$ for every $i \in [\ell]$.

\item [(b)]If $i\in[\ell]$ and $u\in V(G') \setminus V_i$, then $|V_i \setminus N_{G'}(u)|\leq 2\eta^{1/2}n'$.
\end{itemize}
\end{claim}
\begin{proof}[Proof of Claim \ref{claim Vi}]
First, we prove (a). On the one hand, counting the number of edges in $G'$ yields that
\begin{align*}\label{lower bound of edges of G'}
\sum_{1\leq i < j \leq \ell}|V_i| |V_j|\ge e(G') \geq \frac{\delta n'}{2} \overset{\eqref{min-deg-G'}}\ge \left(\frac{\ell-1}{\ell}-\eta\right)\frac{n'^2}{2}.
\end{align*}
On the other hand, we know
\begin{align*}
\sum_{1\leq i < j \leq \ell}|V_i| |V_j|= \frac{1}{2}\sum_{i=1}^\ell|V_i|(n'-|V_i|) = \frac{n'^2}{2} - \frac{1}{2}\sum_{i=1}^\ell |V_i|^2.
\end{align*}
Thus we can conclude that
\begin{equation}\label{upper-bound-square-V}
\sum_{i=1}^\ell |V_i|^2 \leq \frac{n'^2}{\ell} + \eta n'^2.
\end{equation}
Fix an integer $k\in [\ell]$, let  $\theta_k = |V_k|-n'/\ell$. Then
\begin{align}\label{lower-bound-square-V}
\sum_{i=1}^\ell |V_i|^2 &= \left(\frac{n'}{\ell} + \theta_k \right)^2 + \sum_{i\neq k}|V_i|^2  \notag \\
&\geq \left(\frac{n'}{\ell}+\theta_k \right)^2 + \frac{(\sum_{i \neq k}|V_i|)^2}{\ell - 1}    \notag \\
&\geq \left(\frac{n'}{\ell}+\theta_k \right)^2 + \frac{(n'(1-1/\ell)-\theta_k)^2}{\ell-1}      \notag \\
&\geq \frac{n'^2}{\ell}+\theta_k^2.
\end{align}
Combining \eqref{upper-bound-square-V} and \eqref{lower-bound-square-V}, we obtain that
$|\theta_k| \leq \eta^{1/2}n',$
and thus
\[
\left||V_k|- \frac{1}{\ell}n'\right| \leq \eta^{1/2}n'
\] for every $k \in [\ell]$.

Now we prove (b). Fix $k\in [\ell]$ and $k \neq i$. Set $V_0 = V(G') - V_k - V_i$. Combining \eqref{min-deg-G'} and (a), we have for each $u \in V_k$,
\begin{align*}
\left|N_{G'[V_i]}(u)\right| &= d_{G'}(u) - \left|N_{G'[V_0]}(u)\right|  \\
&\geq \left(\frac{\ell-1}{\ell}-\eta\right)n' - (\ell-2)\left(\frac{1}{\ell}+\eta^{1/2}\right)n'  \\
&\geq |V_i| - 2\eta^{1/2}n',
\end{align*}
which implies that $|V_i \setminus N_{G'}(u)|\leq 2\eta^{1/2}n'$.
\end{proof}

It follows from Claim \ref{claim Vi} that $G'$ is close to the $T(n,\ell)$. So the following is obvious by a greedy choice.

\begin{claim}\label{Kr-2}
For $1\leq i<j\leq \ell$, if $e=x_ix_j$ is a non-edge for $x_i \in V_i$ and $x_j \in V_j$, then $G'[N(x_i) \cap N(x_j)]$ contains at least one copy of $K_{r-2}$.
\end{claim}
\begin{proof}[Proof of Claim \ref{Kr-2}]
Without loss of generality, let $e=x_1x_2$ be a non-edge for $x_1 \in V_1$ and $x_2 \in V_2$. Clearly, we can choose a vertex $x_3\in V_3$ such that $x_3\in N(x_1) \cap N(x_2)$ by Claim \ref{claim Vi}. Suppose that $x_3,\dots, x_{p}$ have been picked for $x_i \in V_i$ with $i \in\{3, 4, \dots, p\}$. Then by Claim \ref{claim Vi}, we obtain
\begin{align}
&|N_{G'[V_{p+1}]}(x_1)\cap \dots \cap N_{G'[V_{p+1}]}(x_{p})| \notag \\
&\geq \sum_{i=1}^{p}d_{G'[V_{p+1}]}(x_i)-(p-1)|V_{p+1}|                                                    \notag \\
&\geq p \left(\left(\frac{1}{\ell}-\eta^{1/2}\right)n'-2\eta^{1/2}n'\right)-(p-1)\left(\frac{1}{\ell}+\eta^{1/2}\right)n'  \notag\\
&\ge\left(\frac{1}{\ell}+\eta^{1/2}-4p\eta^{1/2}\right) n'  \notag\\
&\geq 1.   \notag
\end{align}
Note that the last inequality holds by the choice of   $\eta$ and the fact $p\le r-1$. So we can choose a vertex $x_{p+1}$ in $V_{p+1}$ with $x_{p+1}\in \cap_{i=1}^p N(x_i)$. This completes the proof.
\end{proof}

\begin{claim}\label{T-neq-empty}
If $G'$ is not a complete $\ell$-partite  graph, then $T \neq \emptyset$.
\end{claim}
\begin{proof}[Proof of Claim \ref{T-neq-empty}]
Otherwise, $T = \emptyset$ and $G^\ast=G'$. For a non-edge $x_ix_j$ with $x_i \in V_i$ and $x_j \in V_j$ with $i\neq j$, it follows from Claim \ref{Kr-2} that $N(K_r, G^\ast+x_ix_j)>N(K_r, G^\ast)$. Recall that $G^\ast$ is $K_r$-maximal. There exists a subset $S$ with $|S|=\ell+1$ and $x_i,x_j\in S$ such that the subgraph in $G^\ast+x_ix_j$ induced by $S$ is a copy of $K_{\ell+1}$. By the Pigeonhole Principle, there exists $s\in [\ell]$ such that $|S\cap V_s|\ge 2$, which means there is an edge in $G^\ast[V_s]$, a contradiction.
\end{proof}

In the following, we assume that $T \neq \emptyset$, since otherwise, $G'$ is a desired complete $\ell$-partite  graph by Claim \ref{T-neq-empty}. Recall that $\widehat{G'}$ is the $\ell$-partite complement containing all the non-edges of $G'$ between vertex class $V_i$ and $V_j$ for $1\leq i < j\leq \ell$. Now we bound the missing edges in $G'$.

\begin{claim}\label{claim1}
There exists a constant $\lambda_{r, \ell}=\lambda(r, \ell)>0$ such that $e(\widehat{G'}) \leq \lambda_{r, \ell}\varepsilon n^{2}$.
\end{claim}
\begin{proof}[Proof of Claim \ref{claim1}]
On the one hand, we know
\begin{align}
e(G^\ast) &= e(T) + e(T,V(G')) + e(G') \notag \\
& \leq |T|\left(\frac{\ell-1}{\ell}-\eta\right)n + t_2(n-t,\ell) - e(\widehat{G'})\notag\\
& \leq |T|\left(\frac{\ell-1}{\ell}-\eta\right)n + t_2(n,\ell) - e(\widehat{G'}). \notag
\end{align}
On the other hand, we have $e(G^\ast)\ge t_2(n,\ell)-\varepsilon n^2$ by \eqref{lower-bound-r-2-H} and $|T| \leq 10^6 \ell^3(3r-4)^3 \varepsilon n$ by  Claim \ref{T}. Thus,
\begin{align}
e(\widehat{G'})\le  |T|\left(\frac{\ell-1}{\ell}-\eta\right)n +\varepsilon n^{2}&\leq 10^6\ell^3(3r-4)^3\varepsilon n^2. \notag
\end{align}
\end{proof}

By the proof of Claim \ref{T-neq-empty} and the fact that  $G^\ast$ is $K_r$-maximal, we know that adding any non-edge $e \in E(\widehat{G'})$ to $G'$ must form a copy of $K_{\ell+1}$ with some  $t$ vertices in $T$ for $t \in [\ell - 1]$. For each $e \in E(\widehat{G'})$ and $t \in [\ell - 1]$, we say $e$ has {\em property $t$} if adding $e$ to $G^\ast$ forms a copy of $K_{\ell + 1}$ with exactly $t$ vertices in $T$. Let $E_{t}$ denote the collection non-edges in $\widehat{G'}$ of property $t$ and $\mathcal{C}_{t}$ denote the collection of copies of $K_{t}$ in $G^\ast[T]$. Then $E(\widehat{G'})=\cup_{i=1}^{\ell -1}E_{t}$.

For $t \in [\ell - 1]$, we find a subset $\mathcal{Z}(t)$ with small size and covering all the edges in $E_t$. This implies that the subgraph of $G^\ast$ induced by $V(G^\ast)-T-\cup_{i=1}^{\ell -1}\mathcal{Z}(t)$ is a complete $\ell$-partite  graph.  For $t \in [\ell - 1]$, let $\mathcal{C}_{t}$ be the collection of cliques of size $t$ in $T$, and define
\[
 \mathcal{B}_{t} = \{G'[N_{G^\ast}(K) \cap V(G')] : K \in \mathcal{C}_{t}\}.
\]
Note that each graph in the collection of $\mathcal{B}_{t}$ is $K_{\ell + 1 - t}$-free and every non-edge of $E_{t}$ between $V_{i}$ and $V_{j}$ is $(\ell + 1 - t)$-saturating and both of its end are in some graph of $\mathcal{B}_{t}$. Thus, we can use  Lemma \ref{lem} to  obtain a set $\mathcal{Z}_{ij}(t) \subseteq V_{i} \cup V_{j}$ that covers every $(\ell + 1)$-saturating  edge between $V_{i}$ and $V_{j}$ with  property $t$ and
\begin{align}\label{The Lower Bound of I}
|I_{\widehat{G'}}(\mathcal{Z}_{ij}(t))| \geq c^{'}_{\ell,t}|\mathcal{C}_{t}|^{\frac{-1}{\ell - t}}|\mathcal{Z}_{ij}(t)|^{\frac{\ell + 1 - t}{\ell - t}}.
\end{align}


By Claim $\ref{claim1}$ we have
\begin{align}\label{upper bound of I}
|I_{\widehat{G'}}(\mathcal{Z}_{ij}(t))| \leq e(\widehat{G'}) \leq \lambda_{r, \ell}\varepsilon n^{2}.
\end{align}
Note that the number of copies of $K_t$ in $\mathcal{C}$ is  at most
\begin{align}\label{Ct}
|\mathcal{C}_{t}| \leq \binom{|T|}{t} \leq |T|^t \leq \left(10^6\ell^3(3r-4)^3\varepsilon n\right)^t.
\end{align}
Combining \eqref{The Lower Bound of I}, \eqref{upper bound of I} and \eqref{Ct}, we conclude that
\begin{align*}
|\mathcal{Z}_{ij}(t)|^{\frac{\ell + 1 - t}{\ell - t}} \leq  (c'_{\ell, t})^{-1}\lambda_{r, \ell}\varepsilon n^{2}|\mathcal{C}_{t}|^{\frac{1}{\ell - t}}
\leq  (c'_{\ell, t})^{-1}\lambda_{r,\ell}\varepsilon n^2\left(10^6\ell^3(3r-4)^3\varepsilon n \right)^{\frac{t}{\ell-t}}.
\end{align*}
This implies that
\[|\mathcal{Z}_{ij}(t)| \leq C_{\ell, t}(\varepsilon^{\frac{\ell}{\ell + 1 - t}}n^{\frac{\ell - 1}{\ell + 1 - t}})n,
\]
where
\[
C_{\ell, t} = \left[(c'_{\ell, t})^{-1}\lambda_{r,
\ell} \left(10^6\ell^3(3r-4)^3\right)^{\frac{t}{\ell-t}}\right]^{\frac{\ell - t}{\ell - t + 1}}.
\]
Let
\[
\mathcal{Z} = \cup_{t = 1}^{\ell - 1}\cup_{1\le i <j \le \ell}\mathcal{Z}_{ij}(t).
\]
Then $\mathcal{Z}$ covers every non-edge between the parts $V_{1},\ldots,V_{\ell}$, and
\begin{align*}
|\mathcal{Z}| \leq & \sum_{t = 1}^{\ell - 1}\sum_{i \neq j \in [\ell]}C_{\ell, t}(\varepsilon^{\frac{\ell}{\ell + 1 - t}}n^{\frac{\ell - 1}{\ell + 1 - t}})n
\\ \leq & \max_{t \in[\ell - 1]} \left\{(\ell - 1)\binom{\ell}{2} C_{\ell, t}\left(\varepsilon^{\frac{\ell}{\ell + 1 - t}}n^{\frac{\ell - 1}{\ell + 1 - t}}\right)n \right\}
\\ \leq & C_{\ell}\varepsilon n^{\frac{\ell - 1}{\ell}}n,
\end{align*}
where
\[
C_{\ell} = \max_{t \in[\ell - 1]} \left\{(\ell - 1)\binom{\ell}{2} C_{\ell, t}\right\}.
\]

Finally, let $G^\circ=G^\ast - \mathcal{Z} - T$. Then $G^\circ$ is a complete $\ell$-partite graph with
\[
v(G^\circ)\ge \left(1 - C_{\ell}\varepsilon n^{1 - 1/\ell} - 10^6 \ell^3(3r-4)^3 \varepsilon \right)n.
 \]
 By Lemma \ref{Maximal Kl if and only if maximal Kr}, we have a complete $\ell$-partite subgraph of $\mathcal{H}$, which completes the proof.
\end{proof}

\subsection{Tightness of Theorem~\ref{th2}}\label{sub3.2}

In this subsection we show the bound on the number of edges in Theorem \ref{th2} is tight using a construction given by Popielarz, Sahasrabudhe and Snyder \cite{popielarz2018stability}. Now we introduce it. Let $\ell, s, s_1, \dots, s_{\ell-1} \geq 2$ be integers. First, we define a sequence of graphs $G_{2,s_1}$, $G_{3,s_1,s_2}$, $\dots$, $G_{\ell, s_1,\dots,s_{\ell-1}}$ inductively, where $G_{i,s_1,\dots,s_{i-1}}$ will be an $i$-partite graph. First, $G_{2, s_1}$ is the complete bipartite graph $K_{s_1, s_1}$. Let $2 \leq t \leq \ell -1$ and assume that we have defined a $t$-partite graph $G_{t, s_1, s_2,\dots,s_{t-1}}$. Then, we define $G_{t+1,s_1,\dots,s_t}$ as follows:  Let $H_1, \dots, H_{s_t}$ be vertex disjoint copies of $G_{t, s_1, s_2,\dots,s_{t-1}}$ and suppose that $H_p$ has vertex sets $A_1^p,\dots,A_t^p$ for each $p \in [s_t]$. Define $G_{t+1,s_1,\dots,s_t}$ to be the $(t+1)$-partite graph with the first $t$ vertex sets defined as $A_i := A_i^1 \cup \dots \cup A_i^{s_t}$ for $i \in [t]$, and with the $(t+1)$-th vertex set defined as a collection of new vertices $A_{t+1} = \{x_1, \dots, x_{s_t}\}$. Let
\[
E(G_{t+1, s_1, \dots,s_t}) = \bigcup_{p=1}^{s_t}\left(E(H_p)\cup \{x_py : y \in H_q, p, q \in [s_t], p\neq q\}\right).
\]

Let $G_{\ell,s} = G_{\ell,2s,s,\dots,s}$ for $s \geq 2$ (See Figure \ref{Gls}). Then $G_{\ell,s}$ has $\frac{s}{s-1}\left(4s^{\ell-1}-3s^{\ell-2}-1\right)\leq 4\frac{s^\ell}{s-1}$ vertices and at most $4(\ell-1)s^\ell$ edges.

\begin{figure}
\begin{center}

\tikzset{every picture/.style={line width=0.75pt}} 

\begin{tikzpicture}[x=0.75pt,y=0.75pt,yscale=-0.8,xscale=0.8]

\draw  [fill={rgb, 255:red, 155; green, 155; blue, 155 }  ,fill opacity=0.45 ] (33.02,246.75) .. controls (28.37,246.72) and (24.61,242.91) .. (24.63,238.26) -- (24.93,146.93) .. controls (24.95,142.28) and (28.73,138.54) .. (33.38,138.58) -- (58.66,138.78) .. controls (63.31,138.82) and (67.07,142.62) .. (67.05,147.27) -- (66.75,238.6) .. controls (66.73,243.25) and (62.95,246.99) .. (58.3,246.96) -- cycle ;
\draw  [fill={rgb, 255:red, 155; green, 155; blue, 155 }  ,fill opacity=0.45 ] (115.51,247.41) .. controls (110.86,247.38) and (107.1,243.57) .. (107.12,238.92) -- (107.42,147.59) .. controls (107.44,142.94) and (111.22,139.2) .. (115.87,139.24) -- (141.15,139.44) .. controls (145.8,139.48) and (149.56,143.28) .. (149.54,147.93) -- (149.24,239.26) .. controls (149.22,243.91) and (145.44,247.65) .. (140.79,247.62) -- cycle ;
\draw    (48.79,214.49) -- (127.64,165.04) ;
\draw    (48.97,166.08) -- (125.57,214.77) ;
\draw  [fill={rgb, 255:red, 0; green, 0; blue, 0 }  ,fill opacity=1 ] (50.02,217.54) .. controls (48.64,217.53) and (47.52,216.4) .. (47.52,215.02) .. controls (47.53,213.64) and (48.65,212.53) .. (50.03,212.54) .. controls (51.41,212.55) and (52.53,213.68) .. (52.52,215.06) .. controls (52.52,216.44) and (51.4,217.55) .. (50.02,217.54) -- cycle ;
\draw  [fill={rgb, 255:red, 0; green, 0; blue, 0 }  ,fill opacity=1 ] (50.81,169.13) .. controls (49.43,169.12) and (48.31,168) .. (48.31,166.61) .. controls (48.32,165.23) and (49.44,164.12) .. (50.82,164.13) .. controls (52.2,164.15) and (53.32,165.27) .. (53.31,166.65) .. controls (53.31,168.04) and (52.19,169.15) .. (50.81,169.13) -- cycle ;
\draw  [fill={rgb, 255:red, 0; green, 0; blue, 0 }  ,fill opacity=1 ] (126.75,218.79) .. controls (125.37,218.78) and (124.25,217.65) .. (124.26,216.27) .. controls (124.26,214.89) and (125.39,213.78) .. (126.77,213.79) .. controls (128.15,213.8) and (129.26,214.93) .. (129.26,216.31) .. controls (129.25,217.69) and (128.13,218.8) .. (126.75,218.79) -- cycle ;
\draw  [fill={rgb, 255:red, 0; green, 0; blue, 0 }  ,fill opacity=1 ] (127.55,168.08) .. controls (126.17,168.07) and (125.05,166.94) .. (125.06,165.56) .. controls (125.06,164.18) and (126.18,163.07) .. (127.56,163.08) .. controls (128.94,163.09) and (130.06,164.22) .. (130.06,165.6) .. controls (130.05,166.98) and (128.93,168.09) .. (127.55,168.08) -- cycle ;
\draw  [fill={rgb, 255:red, 155; green, 155; blue, 155 }  ,fill opacity=0.45 ] (242.56,78.07) .. controls (239.96,78.02) and (237.85,75.87) .. (237.84,73.28) -- (237.73,24.23) .. controls (237.73,21.64) and (239.83,19.57) .. (242.42,19.62) -- (256.52,19.9) .. controls (259.12,19.95) and (261.23,22.09) .. (261.23,24.69) -- (261.34,73.73) .. controls (261.35,76.33) and (259.25,78.39) .. (256.65,78.34) -- cycle ;
\draw  [fill={rgb, 255:red, 155; green, 155; blue, 155 }  ,fill opacity=0.45 ] (288.57,78.96) .. controls (285.97,78.91) and (283.86,76.77) .. (283.85,74.17) -- (283.74,25.12) .. controls (283.74,22.53) and (285.84,20.46) .. (288.43,20.52) -- (302.53,20.79) .. controls (305.13,20.84) and (307.24,22.98) .. (307.24,25.58) -- (307.35,74.63) .. controls (307.36,77.22) and (305.26,79.29) .. (302.66,79.24) -- cycle ;
\draw    (251.25,60.74) -- (295.08,34.53) ;
\draw    (251.2,34.58) -- (294.08,61.39) ;
\draw  [fill={rgb, 255:red, 0; green, 0; blue, 0 }  ,fill opacity=1 ] (251.62,62.39) .. controls (251.03,62.38) and (250.55,61.63) .. (250.54,60.73) .. controls (250.54,59.82) and (251.02,59.1) .. (251.61,59.11) .. controls (252.21,59.12) and (252.69,59.86) .. (252.69,60.77) .. controls (252.69,61.68) and (252.21,62.4) .. (251.62,62.39) -- cycle ;
\draw  [fill={rgb, 255:red, 0; green, 0; blue, 0 }  ,fill opacity=1 ] (251.92,36.24) .. controls (251.32,36.23) and (250.84,35.48) .. (250.84,34.58) .. controls (250.84,33.67) and (251.32,32.94) .. (251.91,32.96) .. controls (252.5,32.97) and (252.98,33.71) .. (252.99,34.62) .. controls (252.99,35.53) and (252.51,36.25) .. (251.92,36.24) -- cycle ;
\draw  [fill={rgb, 255:red, 0; green, 0; blue, 0 }  ,fill opacity=1 ] (294.42,63.56) .. controls (293.83,63.55) and (293.35,62.81) .. (293.35,61.9) .. controls (293.34,60.99) and (293.82,60.27) .. (294.42,60.28) .. controls (295.01,60.29) and (295.49,61.04) .. (295.49,61.94) .. controls (295.5,62.85) and (295.02,63.57) .. (294.42,63.56) -- cycle ;
\draw  [fill={rgb, 255:red, 0; green, 0; blue, 0 }  ,fill opacity=1 ] (294.71,36.17) .. controls (294.12,36.16) and (293.64,35.41) .. (293.64,34.51) .. controls (293.64,33.6) and (294.11,32.87) .. (294.71,32.88) .. controls (295.3,32.9) and (295.78,33.64) .. (295.78,34.55) .. controls (295.79,35.45) and (295.31,36.18) .. (294.71,36.17) -- cycle ;
\draw  [fill={rgb, 255:red, 155; green, 155; blue, 155 }  ,fill opacity=0.45 ] (362,22) .. controls (362,18.24) and (365.04,15.2) .. (368.8,15.2) -- (389.2,15.2) .. controls (392.96,15.2) and (396,18.24) .. (396,22) -- (396,347.4) .. controls (396,351.16) and (392.96,354.2) .. (389.2,354.2) -- (368.8,354.2) .. controls (365.04,354.2) and (362,351.16) .. (362,347.4) -- cycle ;
\draw  [dash pattern={on 4.5pt off 4.5pt}] (226.56,24.57) .. controls (226.56,14.72) and (234.54,6.73) .. (244.4,6.73) -- (297.93,6.73) .. controls (307.79,6.73) and (315.77,14.72) .. (315.77,24.57) -- (315.77,88.12) .. controls (315.77,97.98) and (307.79,105.97) .. (297.93,105.97) -- (244.4,105.97) .. controls (234.54,105.97) and (226.56,97.98) .. (226.56,88.12) -- cycle ;
\draw  [fill={rgb, 255:red, 155; green, 155; blue, 155 }  ,fill opacity=0.45 ] (241.56,195.07) .. controls (238.96,195.02) and (236.85,192.87) .. (236.84,190.28) -- (236.73,141.23) .. controls (236.73,138.64) and (238.83,136.57) .. (241.42,136.62) -- (255.52,136.9) .. controls (258.12,136.95) and (260.23,139.09) .. (260.23,141.69) -- (260.34,190.73) .. controls (260.35,193.33) and (258.25,195.39) .. (255.65,195.34) -- cycle ;
\draw  [fill={rgb, 255:red, 155; green, 155; blue, 155 }  ,fill opacity=0.45 ] (287.57,195.96) .. controls (284.97,195.91) and (282.86,193.77) .. (282.85,191.17) -- (282.74,142.12) .. controls (282.74,139.53) and (284.84,137.46) .. (287.43,137.52) -- (301.53,137.79) .. controls (304.13,137.84) and (306.24,139.98) .. (306.24,142.58) -- (306.35,191.63) .. controls (306.36,194.22) and (304.26,196.29) .. (301.66,196.24) -- cycle ;
\draw    (250.25,177.74) -- (294.08,151.53) ;
\draw    (250.2,151.58) -- (293.08,178.39) ;
\draw  [fill={rgb, 255:red, 0; green, 0; blue, 0 }  ,fill opacity=1 ] (250.62,179.39) .. controls (250.03,179.38) and (249.55,178.63) .. (249.54,177.73) .. controls (249.54,176.82) and (250.02,176.1) .. (250.61,176.11) .. controls (251.21,176.12) and (251.69,176.86) .. (251.69,177.77) .. controls (251.69,178.68) and (251.21,179.4) .. (250.62,179.39) -- cycle ;
\draw  [fill={rgb, 255:red, 0; green, 0; blue, 0 }  ,fill opacity=1 ] (250.92,153.24) .. controls (250.32,153.23) and (249.84,152.48) .. (249.84,151.58) .. controls (249.84,150.67) and (250.32,149.94) .. (250.91,149.96) .. controls (251.5,149.97) and (251.98,150.71) .. (251.99,151.62) .. controls (251.99,152.53) and (251.51,153.25) .. (250.92,153.24) -- cycle ;
\draw  [fill={rgb, 255:red, 0; green, 0; blue, 0 }  ,fill opacity=1 ] (293.42,180.56) .. controls (292.83,180.55) and (292.35,179.81) .. (292.35,178.9) .. controls (292.34,177.99) and (292.82,177.27) .. (293.42,177.28) .. controls (294.01,177.29) and (294.49,178.04) .. (294.49,178.94) .. controls (294.5,179.85) and (294.02,180.57) .. (293.42,180.56) -- cycle ;
\draw  [fill={rgb, 255:red, 0; green, 0; blue, 0 }  ,fill opacity=1 ] (293.71,153.17) .. controls (293.12,153.16) and (292.64,152.41) .. (292.64,151.51) .. controls (292.64,150.6) and (293.11,149.87) .. (293.71,149.88) .. controls (294.3,149.9) and (294.78,150.64) .. (294.78,151.55) .. controls (294.79,152.45) and (294.31,153.18) .. (293.71,153.17) -- cycle ;
\draw  [dash pattern={on 4.5pt off 4.5pt}] (226.56,141.57) .. controls (226.56,131.72) and (234.54,123.73) .. (244.4,123.73) -- (297.93,123.73) .. controls (307.79,123.73) and (315.77,131.72) .. (315.77,141.57) -- (315.77,206.12) .. controls (315.77,215.98) and (307.79,223.97) .. (297.93,223.97) -- (244.4,223.97) .. controls (234.54,223.97) and (226.56,215.98) .. (226.56,206.12) -- cycle ;
\draw  [fill={rgb, 255:red, 155; green, 155; blue, 155 }  ,fill opacity=0.45 ] (242.56,335.07) .. controls (239.96,335.02) and (237.85,332.87) .. (237.84,330.28) -- (237.73,281.23) .. controls (237.73,278.64) and (239.83,276.57) .. (242.42,276.62) -- (256.52,276.9) .. controls (259.12,276.95) and (261.23,279.09) .. (261.23,281.69) -- (261.34,330.73) .. controls (261.35,333.33) and (259.25,335.39) .. (256.65,335.34) -- cycle ;
\draw  [fill={rgb, 255:red, 155; green, 155; blue, 155 }  ,fill opacity=0.45 ] (288.57,335.96) .. controls (285.97,335.91) and (283.86,333.77) .. (283.85,331.17) -- (283.74,282.12) .. controls (283.74,279.53) and (285.84,277.46) .. (288.43,277.52) -- (302.53,277.79) .. controls (305.13,277.84) and (307.24,279.98) .. (307.24,282.58) -- (307.35,331.63) .. controls (307.36,334.22) and (305.26,336.29) .. (302.66,336.24) -- cycle ;
\draw    (251.25,317.74) -- (295.08,291.53) ;
\draw    (251.2,291.58) -- (294.08,318.39) ;
\draw  [fill={rgb, 255:red, 0; green, 0; blue, 0 }  ,fill opacity=1 ] (251.62,319.39) .. controls (251.03,319.38) and (250.55,318.63) .. (250.54,317.73) .. controls (250.54,316.82) and (251.02,316.1) .. (251.61,316.11) .. controls (252.21,316.12) and (252.69,316.86) .. (252.69,317.77) .. controls (252.69,318.68) and (252.21,319.4) .. (251.62,319.39) -- cycle ;
\draw  [fill={rgb, 255:red, 0; green, 0; blue, 0 }  ,fill opacity=1 ] (251.92,293.24) .. controls (251.32,293.23) and (250.84,292.48) .. (250.84,291.58) .. controls (250.84,290.67) and (251.32,289.94) .. (251.91,289.96) .. controls (252.5,289.97) and (252.98,290.71) .. (252.99,291.62) .. controls (252.99,292.53) and (252.51,293.25) .. (251.92,293.24) -- cycle ;
\draw  [fill={rgb, 255:red, 0; green, 0; blue, 0 }  ,fill opacity=1 ] (294.42,320.56) .. controls (293.83,320.55) and (293.35,319.81) .. (293.35,318.9) .. controls (293.34,317.99) and (293.82,317.27) .. (294.42,317.28) .. controls (295.01,317.29) and (295.49,318.04) .. (295.49,318.94) .. controls (295.5,319.85) and (295.02,320.57) .. (294.42,320.56) -- cycle ;
\draw  [fill={rgb, 255:red, 0; green, 0; blue, 0 }  ,fill opacity=1 ] (294.71,293.17) .. controls (294.12,293.16) and (293.64,292.41) .. (293.64,291.51) .. controls (293.64,290.6) and (294.11,289.87) .. (294.71,289.88) .. controls (295.3,289.9) and (295.78,290.64) .. (295.78,291.55) .. controls (295.79,292.45) and (295.31,293.18) .. (294.71,293.17) -- cycle ;
\draw [dash pattern={on 4.5pt off 4.5pt}] (227.56,281.57) .. controls (227.56,271.72) and (235.54,263.73) .. (245.4,263.73) -- (298.93,263.73) .. controls (308.79,263.73) and (316.77,271.72) .. (316.77,281.57) -- (316.77,350.12) .. controls (316.77,359.98) and (308.79,367.97) .. (298.93,367.97) -- (245.4,367.97) .. controls (235.54,367.97) and (227.56,359.98) .. (227.56,350.12) -- cycle ;
\draw  [fill={rgb, 255:red, 0; green, 0; blue, 0 }  ,fill opacity=1 ] (374.8,45.1) .. controls (374.8,43.39) and (376.19,42) .. (377.9,42) .. controls (379.61,42) and (381,43.39) .. (381,45.1) .. controls (381,46.81) and (379.61,48.2) .. (377.9,48.2) .. controls (376.19,48.2) and (374.8,46.81) .. (374.8,45.1) -- cycle ;
\draw  [fill={rgb, 255:red, 0; green, 0; blue, 0 }  ,fill opacity=1 ] (374.8,166.1) .. controls (374.8,164.39) and (376.19,163) .. (377.9,163) .. controls (379.61,163) and (381,164.39) .. (381,166.1) .. controls (381,167.81) and (379.61,169.2) .. (377.9,169.2) .. controls (376.19,169.2) and (374.8,167.81) .. (374.8,166.1) -- cycle ;
\draw  [fill={rgb, 255:red, 0; green, 0; blue, 0 }  ,fill opacity=1 ] (374.8,305.1) .. controls (374.8,303.39) and (376.19,302) .. (377.9,302) .. controls (379.61,302) and (381,303.39) .. (381,305.1) .. controls (381,306.81) and (379.61,308.2) .. (377.9,308.2) .. controls (376.19,308.2) and (374.8,306.81) .. (374.8,305.1) -- cycle ;
\draw    (311,53) -- (377.9,166.1) ;
\draw    (311,53) -- (377.9,305.1) ;
\draw    (311,169) -- (377.9,45.1) ;
\draw    (311,169) -- (377.9,305.1) ;
\draw    (311,306.6) -- (377.9,166.1) ;
\draw    (311,306.6) -- (377.9,45.1) ;
\draw  [fill={rgb, 255:red, 0; green, 0; blue, 0 }  ,fill opacity=1 ] (376.97,214.75) .. controls (376.97,215.44) and (377.53,216) .. (378.22,216) .. controls (378.91,216) and (379.47,215.44) .. (379.47,214.75) .. controls (379.47,214.06) and (378.91,213.5) .. (378.22,213.5) .. controls (377.53,213.5) and (376.97,214.06) .. (376.97,214.75) -- cycle ;
\draw  [fill={rgb, 255:red, 0; green, 0; blue, 0 }  ,fill opacity=1 ] (376.97,234.75) .. controls (376.97,235.44) and (377.53,236) .. (378.22,236) .. controls (378.91,236) and (379.47,235.44) .. (379.47,234.75) .. controls (379.47,234.06) and (378.91,233.5) .. (378.22,233.5) .. controls (377.53,233.5) and (376.97,234.06) .. (376.97,234.75) -- cycle ;
\draw  [fill={rgb, 255:red, 0; green, 0; blue, 0 }  ,fill opacity=1 ] (376.97,254.75) .. controls (376.97,255.44) and (377.53,256) .. (378.22,256) .. controls (378.91,256) and (379.47,255.44) .. (379.47,254.75) .. controls (379.47,254.06) and (378.91,253.5) .. (378.22,253.5) .. controls (377.53,253.5) and (376.97,254.06) .. (376.97,254.75) -- cycle ;
\draw  [fill={rgb, 255:red, 0; green, 0; blue, 0 }  ,fill opacity=1 ] (269.47,231.31) .. controls (269.47,231.76) and (269.91,232.13) .. (270.47,232.13) .. controls (271.02,232.13) and (271.47,231.76) .. (271.47,231.31) .. controls (271.47,230.86) and (271.02,230.5) .. (270.47,230.5) .. controls (269.91,230.5) and (269.47,230.86) .. (269.47,231.31) -- cycle ;
\draw  [fill={rgb, 255:red, 0; green, 0; blue, 0 }  ,fill opacity=1 ] (269.47,244.35) .. controls (269.47,244.8) and (269.91,245.16) .. (270.47,245.16) .. controls (271.02,245.16) and (271.47,244.8) .. (271.47,244.35) .. controls (271.47,243.9) and (271.02,243.54) .. (270.47,243.54) .. controls (269.91,243.54) and (269.47,243.9) .. (269.47,244.35) -- cycle ;
\draw  [fill={rgb, 255:red, 0; green, 0; blue, 0 }  ,fill opacity=1 ] (269.47,257.39) .. controls (269.47,257.84) and (269.91,258.2) .. (270.47,258.2) .. controls (271.02,258.2) and (271.47,257.84) .. (271.47,257.39) .. controls (271.47,256.94) and (271.02,256.57) .. (270.47,256.57) .. controls (269.91,256.57) and (269.47,256.94) .. (269.47,257.39) -- cycle ;
\draw  [fill={rgb, 255:red, 155; green, 155; blue, 155 }  ,fill opacity=0.45 ] (641,22) .. controls (641,18.24) and (644.04,15.2) .. (647.8,15.2) -- (668.2,15.2) .. controls (671.96,15.2) and (675,18.24) .. (675,22) -- (675,347.4) .. controls (675,351.16) and (671.96,354.2) .. (668.2,354.2) -- (647.8,354.2) .. controls (644.04,354.2) and (641,351.16) .. (641,347.4) -- cycle ;
\draw  [fill={rgb, 255:red, 0; green, 0; blue, 0 }  ,fill opacity=1 ] (653.8,39.1) .. controls (653.8,37.39) and (655.19,36) .. (656.9,36) .. controls (658.61,36) and (660,37.39) .. (660,39.1) .. controls (660,40.81) and (658.61,42.2) .. (656.9,42.2) .. controls (655.19,42.2) and (653.8,40.81) .. (653.8,39.1) -- cycle ;
\draw  [fill={rgb, 255:red, 0; green, 0; blue, 0 }  ,fill opacity=1 ] (653.8,160.1) .. controls (653.8,158.39) and (655.19,157) .. (656.9,157) .. controls (658.61,157) and (660,158.39) .. (660,160.1) .. controls (660,161.81) and (658.61,163.2) .. (656.9,163.2) .. controls (655.19,163.2) and (653.8,161.81) .. (653.8,160.1) -- cycle ;
\draw  [fill={rgb, 255:red, 0; green, 0; blue, 0 }  ,fill opacity=1 ] (653.8,299.1) .. controls (653.8,297.39) and (655.19,296) .. (656.9,296) .. controls (658.61,296) and (660,297.39) .. (660,299.1) .. controls (660,300.81) and (658.61,302.2) .. (656.9,302.2) .. controls (655.19,302.2) and (653.8,300.81) .. (653.8,299.1) -- cycle ;
\draw    (590,47) -- (656.9,160.1) ;
\draw    (590,47) -- (656.9,299.1) ;
\draw    (590,163) -- (656.9,39.1) ;
\draw    (590,163) -- (656.9,299.1) ;
\draw    (590,300.6) -- (656.9,160.1) ;
\draw    (590,300.6) -- (656.9,39.1) ;
\draw  [fill={rgb, 255:red, 0; green, 0; blue, 0 }  ,fill opacity=1 ] (655.97,208.75) .. controls (655.97,209.44) and (656.53,210) .. (657.22,210) .. controls (657.91,210) and (658.47,209.44) .. (658.47,208.75) .. controls (658.47,208.06) and (657.91,207.5) .. (657.22,207.5) .. controls (656.53,207.5) and (655.97,208.06) .. (655.97,208.75) -- cycle ;
\draw  [fill={rgb, 255:red, 0; green, 0; blue, 0 }  ,fill opacity=1 ] (655.97,228.75) .. controls (655.97,229.44) and (656.53,230) .. (657.22,230) .. controls (657.91,230) and (658.47,229.44) .. (658.47,228.75) .. controls (658.47,228.06) and (657.91,227.5) .. (657.22,227.5) .. controls (656.53,227.5) and (655.97,228.06) .. (655.97,228.75) -- cycle ;
\draw  [fill={rgb, 255:red, 0; green, 0; blue, 0 }  ,fill opacity=1 ] (655.97,248.75) .. controls (655.97,249.44) and (656.53,250) .. (657.22,250) .. controls (657.91,250) and (658.47,249.44) .. (658.47,248.75) .. controls (658.47,248.06) and (657.91,247.5) .. (657.22,247.5) .. controls (656.53,247.5) and (655.97,248.06) .. (655.97,248.75) -- cycle ;
\draw  [dash pattern={on 4.5pt off 4.5pt}] (529.33,28.73) .. controls (529.33,20.6) and (535.93,14) .. (544.07,14) -- (588.27,14) .. controls (596.4,14) and (603,20.6) .. (603,28.73) -- (603,78.57) .. controls (603,86.7) and (596.4,93.3) .. (588.27,93.3) -- (544.07,93.3) .. controls (535.93,93.3) and (529.33,86.7) .. (529.33,78.57) -- cycle ;
\draw  [dash pattern={on 4.5pt off 4.5pt}] (529.33,133.73) .. controls (529.33,125.6) and (535.93,119) .. (544.07,119) -- (588.27,119) .. controls (596.4,119) and (603,125.6) .. (603,133.73) -- (603,183.57) .. controls (603,191.7) and (596.4,198.3) .. (588.27,198.3) -- (544.07,198.3) .. controls (535.93,198.3) and (529.33,191.7) .. (529.33,183.57) -- cycle ;
\draw  [dash pattern={on 4.5pt off 4.5pt}] (529.33,293.73) .. controls (529.33,285.6) and (535.93,279) .. (544.07,279) -- (588.27,279) .. controls (596.4,279) and (603,285.6) .. (603,293.73) -- (603,343.57) .. controls (603,351.7) and (596.4,358.3) .. (588.27,358.3) -- (544.07,358.3) .. controls (535.93,358.3) and (529.33,351.7) .. (529.33,343.57) -- cycle ;
\draw  [fill={rgb, 255:red, 0; green, 0; blue, 0 }  ,fill opacity=1 ] (565.47,225.31) .. controls (565.47,225.76) and (565.91,226.13) .. (566.47,226.13) .. controls (567.02,226.13) and (567.47,225.76) .. (567.47,225.31) .. controls (567.47,224.86) and (567.02,224.5) .. (566.47,224.5) .. controls (565.91,224.5) and (565.47,224.86) .. (565.47,225.31) -- cycle ;
\draw  [fill={rgb, 255:red, 0; green, 0; blue, 0 }  ,fill opacity=1 ] (565.47,238.35) .. controls (565.47,238.8) and (565.91,239.16) .. (566.47,239.16) .. controls (567.02,239.16) and (567.47,238.8) .. (567.47,238.35) .. controls (567.47,237.9) and (567.02,237.54) .. (566.47,237.54) .. controls (565.91,237.54) and (565.47,237.9) .. (565.47,238.35) -- cycle ;
\draw  [fill={rgb, 255:red, 0; green, 0; blue, 0 }  ,fill opacity=1 ] (565.47,251.39) .. controls (565.47,251.84) and (565.91,252.2) .. (566.47,252.2) .. controls (567.02,252.2) and (567.47,251.84) .. (567.47,251.39) .. controls (567.47,250.94) and (567.02,250.57) .. (566.47,250.57) .. controls (565.91,250.57) and (565.47,250.94) .. (565.47,251.39) -- cycle ;
\draw  [fill={rgb, 255:red, 0; green, 0; blue, 0 }  ,fill opacity=1 ] (421.33,188.83) .. controls (421.33,188) and (420.66,187.33) .. (419.83,187.33) .. controls (419,187.33) and (418.33,188) .. (418.33,188.83) .. controls (418.33,189.66) and (419,190.33) .. (419.83,190.33) .. controls (420.66,190.33) and (421.33,189.66) .. (421.33,188.83) -- cycle ;
\draw  [fill={rgb, 255:red, 0; green, 0; blue, 0 }  ,fill opacity=1 ] (441.33,188.83) .. controls (441.33,188) and (440.66,187.33) .. (439.83,187.33) .. controls (439,187.33) and (438.33,188) .. (438.33,188.83) .. controls (438.33,189.66) and (439,190.33) .. (439.83,190.33) .. controls (440.66,190.33) and (441.33,189.66) .. (441.33,188.83) -- cycle ;
\draw  [fill={rgb, 255:red, 0; green, 0; blue, 0 }  ,fill opacity=1 ] (458.33,188.83) .. controls (458.33,188) and (457.66,187.33) .. (456.83,187.33) .. controls (456,187.33) and (455.33,188) .. (455.33,188.83) .. controls (455.33,189.66) and (456,190.33) .. (456.83,190.33) .. controls (457.66,190.33) and (458.33,189.66) .. (458.33,188.83) -- cycle ;
\draw   (211.27,32.44) .. controls (206.6,32.46) and (204.28,34.8) .. (204.3,39.47) -- (204.89,181.52) .. controls (204.92,188.19) and (202.6,191.53) .. (197.93,191.55) .. controls (202.6,191.53) and (204.94,194.85) .. (204.97,201.52)(204.96,198.52) -- (205.57,344.14) .. controls (205.59,348.81) and (207.93,351.13) .. (212.6,351.11) ;
\draw   (521.93,30.44) .. controls (517.26,30.46) and (514.94,32.8) .. (514.96,37.47) -- (515.56,179.52) .. controls (515.59,186.19) and (513.27,189.53) .. (508.6,189.55) .. controls (513.27,189.53) and (515.61,192.85) .. (515.64,199.52)(515.63,196.52) -- (516.24,342.14) .. controls (516.26,346.81) and (518.6,349.13) .. (523.27,349.11) ;

\draw (71.67,260) node [anchor=north west][inner sep=0.75pt]  [font=\tiny] [align=left] {$K_{2s,2s}$};
\draw (73.67,386) node [anchor=north west][inner sep=0.75pt]  [font=\normalsize] [align=left] {$G_{2,s}$};
\draw (257,87) node [anchor=north west][inner sep=0.75pt]  [font=\tiny] [align=left] {$G_{2,s}$};
\draw (257,205) node [anchor=north west][inner sep=0.75pt]  [font=\tiny] [align=left] {$G_{2,s}$};
\draw (257,346) node [anchor=north west][inner sep=0.75pt]  [font=\tiny] [align=left] {$G_{2,s}$};
\draw (288,386) node [anchor=north west][inner sep=0.75pt]  [font=\normalsize] [align=left] {$G_{3,s}$};
\draw (567,386) node [anchor=north west][inner sep=0.75pt]  [font=\normalsize] [align=left] {$G_{\ell,s}$};
\draw (542,47) node [anchor=north west][inner sep=0.75pt]  [font=\tiny] [align=left] {$G_{\ell-1,s}$};
\draw (542,155) node [anchor=north west][inner sep=0.75pt]  [font=\tiny] [align=left] {$G_{\ell-1,s}$};
\draw (542,313) node [anchor=north west][inner sep=0.75pt]  [font=\tiny] [align=left] {$G_{\ell-1,s}$};
\draw (180.67,185.67) node [anchor=north west][inner sep=0.75pt]  [font=\normalsize] [align=left] {$s$};
\draw (490.67,185.67) node [anchor=north west][inner sep=0.75pt]  [font=\normalsize] [align=left] {$s$};

\end{tikzpicture}

\end{center}
\caption{The construction of $G_{\ell,s}$.}
\label{Gls}
\end{figure}
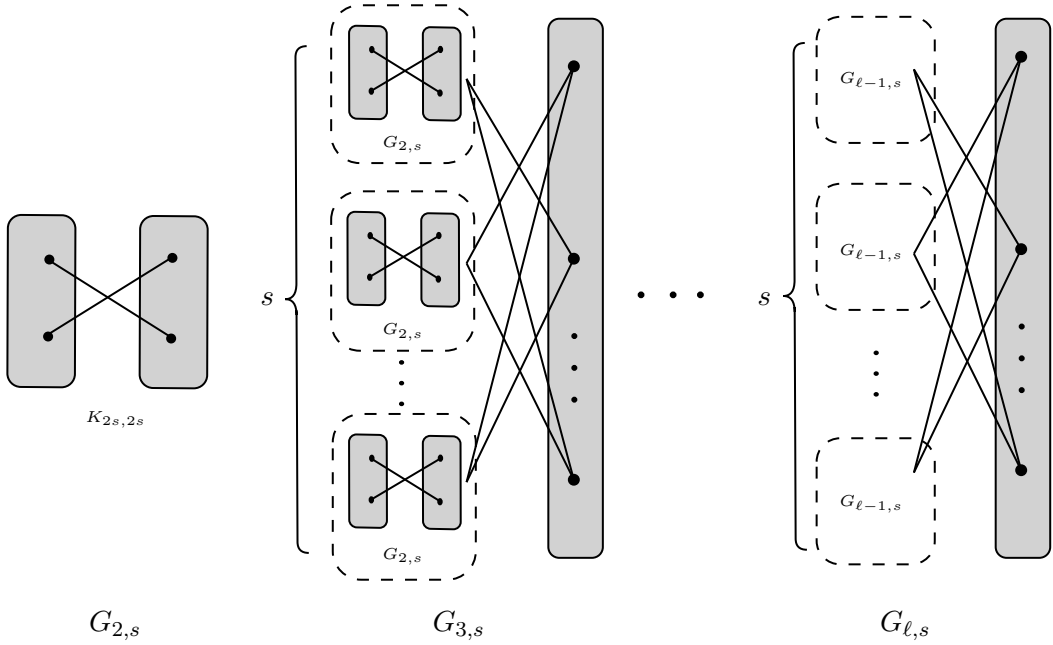

Then, for an integer $n$ with $4s^{\ell} \ell + s\le n\le 4s^{\ell} \ell + 2s$,  we construct a  $\ell$-partite graph $H_{\ell,s}(n)$ with $n$-vertices. Let $H_1,\dots,H_s$ be vertex disjoint copies of $G_{\ell,s}$ with vertex partitions $H_p = A_1^p \cup \dots \cup A_\ell^p$ for each $p \in [s]$. For $i\in[\ell]$, choose a positive integer $m_i$ satisfying
\[
\sum_{p=1}^t |A_i^p|+ m_i \in \left\{\left\lfloor\frac{n-s}{\ell}\right\rfloor, \left\lceil\frac{n-s}{\ell}\right\rceil \right\} \text{ and } \sum_{i=1}^{\ell}\left(\sum_{p=1}^s |A_i^p|+ m_i \right) = n-s.
\]
Let $A_i = A_i^1 \cup \dots \cup A_i^s \cup Y_i$, where
$Y_i$ is a collection of $m_i$ new vertices for $i \in [\ell]$, and let $A_{\ell+1} = \{x_1,\dots,x_s\}$ be a collection of $s$ new vertices. The vertex set of $H_{\ell,s}(n)$ is $\cup_{i=1}^{\ell+1}A_i$. The edge set is defined as follows: For $i\in[s]$, $x_i$ is joined to $V(H_i)$, and for $i,j \in[\ell]$, $x\in A_i$, $y\in A_j$, $xy$ is an edge if and only if $i\neq j$ and $xy$ is not in any of the graphs $H_1,\dots,H_t$.

Popielarz, Sahasrabudhe and Snyder \cite{popielarz2018stability} showed that $H_{\ell,s}(n)$ has the following properties.

\begin{proposition}[Popielarz, Sahasrabudhe and Snyder \cite{popielarz2018stability}]\label{independent set}
The graph $H_{\ell,s}(n)$ has the following properties.\\
1. Any induced complete $\ell$-partite subgraph of $H_{\ell,s}(n)$ has at most $n - 2s^{\ell}$ vertices. \\
2. If we embed a copy of any $K_{\ell+1}$-saturated graph on $A_{\ell+1}$, then the resulting graph is also a $K_{\ell+1}$-saturated graph.
\end{proposition}

\begin{theorem}
For ${\ell} \geq r \geq 3$, there exists a real number $c=c(\ell)$,  an integer $n$ and a $\mathcal{K}_{\ell+1}^{r}$-saturated $r$-graph $\mathcal{H}$ on $n$ vertices with $\left |\mathcal{H}\right| \geq t_{r}(n,\ell) - Cn^{r-1+1/\ell}$ for some $C=C(c)$ such that any complete $\ell$-partite subgraph of $\mathcal{H}$ has at most $(1 - c)n$ vertices.
\end{theorem}
\begin{proof}
For sufficiently large $s$, let $H_{\ell,s}(n)$ be the $(\ell+1)$-partite graph constructed above with $V(H_{\ell,s}(n))=\cup_{i=1}^{\ell+1}A_i$ and $A_{\ell+1}=\{x_1,\ldots,x_s\}$. Let $\mathcal{H}_0$ be a $\mathcal{K}_{\ell+1}^{r}$-saturated $r$-graph on $s$ vertices. We embed a copy of $\partial_{r-2} \mathcal{H}_0$ to $A_{\ell+1}$ in $H_{\ell,s}(n)$, and let the resulting graph be $G$. Then $G[A_{\ell+1}]$ is $K_r$-maximal, $K_{\ell+1}$-free by Lemma \ref{Maximal Kl if and only if maximal Kr}.

First, we show that $G$ is $K_r$-maximal, $K_{\ell+1}$-free. Note that if $x$ is a common neighbor of $x_i$ and $x_j$, then $x\in A_{\ell+1}$. This together with Proposition \ref{independent set} yields that $G$ is $K_{\ell+1}$-free. Let $E'\subseteq E(\overline{G})$ satisfy   $N(K_r, G+E')>N(K_r, G)$. If $E'\cap E({\overline{H_{\ell,s}(n)}})\neq \emptyset$, then $G+E'$ contains a copy of $K_{\ell+1}$ by Proposition \ref{independent set}. Otherwise, $E'\subseteq  E(\overline{G[A_{\ell+1}]})$, which implies that $G+E'$ contains a copy of $K_{\ell+1}$ in $A_{\ell+1}$. So,  $G$ is $K_r$-maximal, $K_{\ell+1}$-free.

Next, we show that each edge $uv\in E(G)$ is contained in a copy of $K_r$. It is obvious if   $u,v\in A_{\ell+1}$ by the construction of $G[A_{\ell+1}]$. Suppose that $u\in V_i$ and $v\in V_j$ for some $i,j\in [\ell]$ and $i\neq j$. Let $Z=\cup_{1\le a\le \ell, a\neq i,j}Y_a$. Then $G[\{u,v\}\cup Z]$ is a complete $\ell$-partite graph and we are done. Suppose that $u\in A_{\ell+1}$ and $v\notin  A_{\ell+1}$. Then we can get a copy of $K_r$ by finding the common neighbors greedily.

Then, we define the final $r$-graph $\mathcal{H}$ by $V(\mathcal{H}) = V(G)$ and
\[
\mathcal{H} = \left\{\{u_1, u_2, \dots, u_r\}: \{u_1, u_2, \dots, u_r\}\,\,{\rm forms}\,\, {\rm a \, copy \,of }\,\, K_r \,\, {\rm in} \,\, G\right\}.
\]
The above arguments shows that $\partial_{r-2} \mathcal{H}=G$ and  $\mathcal{H}$ is $\mathcal{K}_{\ell+1}^{r}$-saturated by Lemma \ref{Maximal Kl if and only if maximal Kr}.

Our goal is to calculate the lower bound of $|\mathcal{H}|$. So we need to calculate the number of missing $K_r$. By the construction, it is easy to see that the  number of non-edges between $A_1$ and $A_2$ is $4s^{\ell+1}$ in $G$, and $e(\widehat{G}[A_1,A_2]) \geq e(\widehat{G}[A_i,A_j])$ for each  $i,j\in[\ell]$ and $i\neq j$. Moreover, for each non-edges    $uv$ with $u\in V_i$ and $v\in V_j$, there are at most $\binom{\ell-1}{r-2}\left(\frac{n-s}{\ell}\right)^{r-2}$ copies of $K_r$ containing $uv$ and  intersecting at most one vertex in each $V_k$ for $k\in [\ell+1]$. For simplicity, let  $s = cn^{1/\ell}$ as $4s^{\ell} \ell + s\le n\le 4s^{\ell} \ell + 2s$. Then
\begin{align}
\mathcal{|H|} \geq & \binom{\ell}{r}\left(\frac{n-s}{\ell}\right)^r - \binom{\ell+1}{2} 4s^{\ell+1}\binom{\ell-1}{r-2}\left(\frac{n-s}{\ell}\right)^{r-2}-\binom{s}{r}             \notag\\
\geq & t_r(n,\ell)- \binom{\ell}{r}\frac{rs}{\ell^r}n^{r-1} - \binom{\ell+1}{2}4s^{\ell+1}\binom{\ell-1}{r-2}\frac{n^{r-2}}{\ell^{r-2}}\left(1-\frac{s}{n}\right)^{r-2}- \binom{s}{r}   \notag\\
\geq & t_r(n,\ell)- \binom{\ell}{r}\frac{rs}{\ell^r}n^{r-1} - \binom{\ell+1}{2}4s^{\ell+1}\binom{\ell-1}{r-2}\frac{n^{r-2}}{\ell^{r-2}}\left(1-(r-2)\frac{s}{n}\right)-n^{r/\ell}  \notag\\
\geq & t_r(n,\ell)- \left(\binom{\ell}{r}\frac{cr}{\ell^r} + \frac{4c^{\ell+1}\binom{\ell+1}{2}\binom{\ell-1}{r-1}}{\ell^{r-2}}+1\right)n^{r-1+1/\ell}.   \notag
\end{align}

By Lemma $\ref{Maximal Kl if and only if maximal Kr}$ and Proposition \ref{independent set}, any complete $\ell$-partite subgraph of $G$ has at most $n - 2c^{\ell}n$ vertices, which completes the proof.
\end{proof}

\section{Positive co-degree stability}\label{SEC:4}

\subsection{Proof of Theorem~\ref{main result}}\label{SEC:proof}

In this subsection, we will prove Theorem~\ref{main result}. Let $\ell \geq 3$ be an integer and
\[
f(\ell)= \begin{cases}
2/7 \quad & \ell=3,\\
\frac{3\ell-7}{3\ell-1} \quad & \ell \geq 4.
\end{cases}
\]
Suppose that $\mathcal{H}$ is a $\mathcal{K}_{\ell+1}^3$-saturated $3$-graph on $n$ vertices with $\delta_{2}^{+}(\mathcal{H}) > f(\ell) n$. By Lemma \ref{Maximal Kl if and only if maximal Kr}, $\partial\mathcal{H}$ is $K_3$-maximal, $K_{\ell+1}$-free and $t^{+}(\partial\mathcal{H}) > f(\ell) n$. Now we add some non-edges to $\partial\mathcal{H}$ to get a  $K_{\ell+1}$-saturated supergraph $G$ of $\partial\mathcal{H}$. It follows from $\partial\mathcal{H}$ is $K_3$-maximal, $K_{\ell+1}$-free that $N(K_3, G)=N(K_3, \partial\mathcal{H})$. So $t^{+}(G) = t^{+}(\partial\mathcal{H}) > f(\ell) n$. If $G$ is $\ell$-partite, then $\mathcal{H}$  is also $\ell$-partite by Observation \ref{l-partite}. Otherwise, $G$ contains a copy of $5$-wheel-like subgraph by Lemma~\ref{Lemma:maximal}.

We choose  a $5$-wheel-like subgraph $W_{\ell,k}$ of $G$ with the top $v$ and the bottom $u_1u_2$ such that the order $s = 2\ell - k+1$ of $W_{\ell,k}$  is minimum. Note that $0\leq k\leq \ell-2$. Let $R = V(Q_1) \cap V(Q_2)$, and let $X$ be the set of  the vertices in $G$ that are adjacent to all vertices in $R$. Then $V(W_{\ell, k})\setminus R\subseteq X$. For a vertex $w\in V(G)$ and a subset $S\subseteq V(G)$, let $e(w, S)=|S\cap N_G(v)|$. The following claim is easy.

\begin{claim}\label{Claim:vertex}
For  each vertex $w \in X$, $e(w, V(W_{\ell,k}))\le s-3$.
\end{claim}

\begin{proof}
First, we consider the case that $w\in V(W_{\ell, k})$. We notice that $G$ is $K_{\ell+1}$-free,  $vu_i\notin  E(G)$ and $e(u_i, V(Q_{3-i})\setminus R)\le |V(Q_{3-i})\setminus R|-1$ for $i=1,2$. This means $e(w, V(W_{\ell,k}))\le s-3$ if $w\in \{u_1, u_2, v\}$. Let $w\in V(Q_1)\setminus R$ or $ V(Q_2)\setminus R$, say $w\in V(Q_{1})\setminus R$. Note that $e(w, V(Q_{2})\setminus R)\le |V(Q_{2})\setminus R|-1$. If $e(w, V(W_{\ell,k})) = s-2$, then we can remove the vertices not adjacent to $w$ in $W_{\ell,k}$ and add $w$ to $R$. The resulting graph $W'$ is still a $5$-wheel-like graph with $|W'|<s$, a contradiction to the choice of $W_{\ell,k}$. So we have $e(w, V(W_{\ell,k}))\le s-3$.

Then we consider the case that $w\notin V(W_{\ell, k})$. It follows from $G$ is $K_{\ell+1}$-free that
\begin{align}\label{number-common-neighbor}
\max\left\{e(w, V(Q_i)\cup\{u_i\}), e(w, V(Q_i)\cup\{v\})\right\}\le \ell-1
\end{align}
for $i=1,2$. If $vw\notin E(G)$, then  we are done by \eqref{number-common-neighbor}. Let $vw\in E(G)$. If $e(w, V(W_{\ell,k}))\ge s-2$, then $e(w, V(Q_i))=\ell-2$ and $u_iw \in E(G)$ by \eqref{number-common-neighbor} for $i=1,2$. In this case, we can remove the vertices not adjacent to $w$ in $W_{\ell,k}$ and add $w$ to $W_{\ell,k}$. The resulting graph $W'$ is still a $5$-wheel-like graph with $|W'|<s$, a contradiction to the choice of $W_{\ell,k}$. This completes the proof of Claim \ref{Claim:vertex}.
\end{proof}

For $w\in V(G)$, we define that $w$ is not adjacent to itself. By the argument in the proof of Claim~\ref{Claim:vertex}, we have that if $w \in X$, then $w$ must be not adjacent to the vertices in one of the following eight cases.
\begin{align}\label{eight-case}
&(1): \{v,u_1,u_2\},(2): \{v,u_1,q_2\},(3): \{v,u_2,q_1\},(4): \{u_1,q_1,q_2\},\notag\\
&(5): \{u_2,q_1,q_2\},(6): \{q_1,q_1',q_2\},(7): \{q_1,q_2,q_2'\},(8): \{v,q_1,q_2\},
\end{align}
where $q_i$ and $q_i'$ are any two different vertices in $V(Q_i)\setminus R$ for $i=1 ,2$.

For $x\in V(G)$ and $e\in E(G)$, we say $x$ \emph{controls}  $e$ if $x$ is adjacent to the two endpoints of edge $e$, otherwise, we say it \emph{misses} $e$.  For $E'\subseteq E(G)$, we use $t(E')$ to denote the number of pairs $(e, w)$ with $e\in E'$ and $w\in V(G)$ such that $w$ controls $e$.

Let $E_1$ be the edge set of $W_{\ell,k}$. Then
\[
|E_1| = \ell^2 +\ell -\frac{1}{2}k^2-\frac{1}{2}k-1.
\]
We bound the number $t(E_1)$  by double counting. Note that if $wu\notin E(G)$ for some $u\in V(W_{\ell,k})$, then $w$ misses at least $d_{W_{\ell,k}}(u) $ edges in $E_1$. Thus, if $w \notin X$, there exists $u\in R$ with $d_{W_{\ell,k}}(u) =s-1=2\ell-k$ such that $wu\notin E(G)$. Then $w$ misses at least $2\ell -k$ edges in $E_1$. This implies that the number of pairs $(e, w)$ with $e\in E_1$ and $w\in V(G)\setminus X$  is at most
\begin{align}\label{number-tri-not-X}
(n-|X|)(|E_1|-(2\ell -k)).
\end{align}
Suppose that $w\in X$. Then $V(W_{\ell,k})$ has at least three vertices $\{x,y,z\}$ not adjacent to $w$, and $\{x,y,z\}$ belongs to one of the cases listed in \eqref{eight-case}. Note the number of missing edges in $E_1$ of $w$ is the number of edges in $E_1$ adjacent to the vertices in $\{x,y,z\}$, which is at least $3\ell-1$ in case (1)--(7). In case (8), if $k \leq \ell-3$, then $w$ misses at least $4\ell - k -4\geq 3\ell - 1$ edges in $E_1$. Otherwise, $wu_1\notin E(G)$ or $wu_2\notin E(G)$ (since otherwise, $\{w, u_1, u_2\}\cup R$ forms a copy of $K_{\ell+1}$), which also implies $w$ misses at least $3\ell - 1$ edges in $E_1$. Thus,  the number of pairs $(e, w)$ with $e\in E_1$ and $w\in X$  is at most
\begin{align}\label{number-tri-in-X}
|X|(|E_1|- (3\ell -1)).
\end{align}

Combining \eqref{number-tri-not-X} and \eqref{number-tri-in-X}, we have
\begin{align}\label{number-tri-in-E1}
t(E_1)\le |X|(|E_1|- (3\ell -1))+(n-|X|)(|E_1|-(2\ell -k)).
\end{align}
Note that if $k\ge 1$, then every edge $e\in E_1$ is contained in a triangle, which implies that $e$ is contained in at lest  $t^{+}(G)$ pairs $(e,W)$. Thus,
\[
t(E_1)\ge |E_1|t^{+}(G),
\]
which together with \eqref{number-tri-in-E1} yields that if $k\ge 1$, then
\begin{align}\label{align2}
t^{+}(G) \leq  \frac{|X|(|E_1|- (3\ell -1))+(n-|X|)(|E_1|-(2\ell -k))}{|E_1|}.
\end{align}

Suppose that $k\ge 1$. Let $E_2$ be the set of edges in $W_{\ell,k}$ with one end in $R$ and the other in $V(W_{\ell,k})\setminus R$. Then
\[
|E_2|= 2\ell k-2k^2+k.
\]
We give another bound on  $t^{+}(G)$ by bounding $t(E_2)$. On the one hand, each edge in $E_2$ is contained in a triangle, which implies that
\begin{align}\label{lower-bound-t-E2}
t(E_2)\ge |E_2|t^{+}(G).
\end{align}
On the other hand, if $w\in X$, then there are at least three vertices in $V(W_{\ell,k})$ not adjacent to $w$ by Claim \ref{Claim:vertex}, and so $w$ misses at least $3k$ edges in $E_2$. If $w \notin X$, then there exists $x\in R$ such that $xw\notin E(G)$, which means that $w$ misses at least $s-k=2\ell -2k +1$ edges in $E_2$. Thus,
\begin{align}\label{upper-bound-t-E2}
t(E_2)\le |X|(|E_2|-3k)+(n-|X|)(|E_2|-(2\ell-2k+1)).
\end{align}
Combining \eqref{lower-bound-t-E2} and \eqref{upper-bound-t-E2}, we conclude that
\begin{align}\label{align3}
t^{+}(G) \leq \frac{|X|(|E_2|-3k)+(n-|X|)(|E_2|-(2\ell-2k+1))}{|E_2|}.
\end{align}

\begin{claim}\label{Claim:X}
For $k \geq 2$, we have $|X| \geq n-k(n-t^{+}(G))/2$.
\end{claim}

\begin{proof}
Note that $G[R]$ is a clique and every edge in $G[R]$ is contained in a triangle. So, for every pair $x,y \in R$, $|N_G(x) \cap N_G(y)| \geq t^{+}(G)$. If $k$ is even, then we can partition $R$ into $k/2$ disjoint pairs and so
\begin{align}
|X| \geq t^{+}(G)-\left(\frac{k}{2}-1\right)\left(n-t^{+}(G)\right) = n-\frac{k(n-t^{+}(G))}{2}.\notag
\end{align}

Assume that $k$ is odd. If $k=3$, we may assume that $R = \{x_1,x_2,x_3\}$. For $1 \leq i < j \leq 3$, define $X_{ij} = \left(N_G(x_i)\cap N_G(x_j)\right)\setminus X $. Then $X_{12}, X_{13}, X_{23}$ is pairwise disjoint and $|X_{ij}|+ |X|$ is the number of triangles containing $x_ix_j$, which is at least $t^{+}(G)$. So
\begin{align}
3t^{+}(G) \leq |X_{12}|+|X_{13}|+|X_{23}| + 3|X| \leq n+2|X|,\notag
\end{align}
which implies that
\begin{align}\label{case-k=3}
|X|\geq \frac{3t^{+}(G)-n}{2}.
\end{align}
For the case that $k\ge 5$, by \eqref{case-k=3}, we have
\begin{align}
|X| \geq \left(\frac{3}{2}t^{+}(G) - \frac{1}{2}n\right) - \frac{k-3}{2}\left(n - t^{+}(G)\right) = n-\frac{k(n-t^{+}(G))}{2}, \notag
\end{align}
which completes the proof of Claim \ref{Claim:X}.
\end{proof}

We divide our remaining proof into the following  four cases according to the size of $R$.\\

\textbf{Case 1}. $k =0$.

Let $Y$ be the set of vertices in $G$ that are adjacent to all vertices in $Q_1$. Using the same arguments as in  Claim~\ref{Claim:X}, we have
\[
|Y| \geq n-\frac{(\ell-1)(n-t^{+}(G))}{2}.
\]
Let $E_3=E(W_{\ell,k})\setminus \{u_1u_2\}$. Then $|E_3| = \ell^2 + \ell -2$. For each $w \in V(G)$, if $w \in Y$, then $w$ must be nonadjacent to $v$, $u_1$ and one vertex in $V(Q_2) \cup \{u_2\}$. So $w$ misses at least $4\ell-4$ edges in $E_3$. If $w \notin Y$, then $w$ is not adjacent to at least one vertex of $Q_1$, which implies $w$ misses at least $3\ell -2$ edges in $E_3$ by  Claim~\ref{Claim:vertex}. Therefore, we have
\begin{align}
t^{+}(G) &\leq \frac{|Y|(|E_3|-(4\ell -4))+(n-|Y|)(|E_3|-(3\ell -2))}{|E_3|}\notag\\
&= n- \frac{(3\ell-2)n}{|E_3|}-\frac{(\ell-2)|Y|}{|E_3|}\notag\\
& \leq n- \frac{(3\ell-2)n}{\ell^2 + \ell -2}-\frac{(\ell-2)}{\ell^2 + \ell -2}\left(n-\frac{(\ell-1)(n-t^{+}(G))}{2}\right).\notag
\end{align}
This implies that
\begin{align}
\left(2(\ell^2 + \ell -2)+(\ell-2)(\ell-1)\right)t^{+}(G)   \leq  \left(2(\ell^2 + \ell -2)-2(3\ell-2)-2(\ell-2)+(\ell-2)(\ell-1)\right)n,\notag
\end{align}
i.e.,
\[
(3\ell^2-\ell-2)t^{+}(G)\le (3\ell^2-9\ell+6)n.
\]
So, $t^{+}(G) \leq \frac{3\ell -6}{3\ell+2}n < f(\ell)n$, a contradiction.

\textbf{Case 2}. $k =1$.

It follows from \eqref{align3} that
\begin{align*}
t^{+}(G) \leq \frac{|X|(|E_2|-3)+(n-|X|)(|E_2|-(2\ell-1))}{|E_2|} = \frac{(2\ell-4)}{2\ell-1}|X|,
\end{align*}
which implies that
\begin{align}\label{bound-X-case2}
|X| \geq \frac{2\ell-1}{2\ell-4}t^{+}(G).
\end{align}
Substituting \eqref{bound-X-case2} into \eqref{align2} yields that
\begin{align*}
t^{+}(G) &\leq \frac{|X|(|E_1|- (3\ell -1))+(n-|X|)(|E_1|-(2\ell -1))}{|E_1|}\notag\\
&= n-\frac{(2\ell-1)n}{|E_1|}-\frac{\ell|X|}{|E_1|} \notag\\
&\le  \frac{\ell^2-\ell-1}{\ell^2+\ell-2}n-\frac{\ell(2\ell-1)}{(\ell^2+\ell-2)(2\ell-4)}t^{+}(G).
\end{align*}
Thus,
\[
t^{+}(G)\le \frac{(\ell^2-\ell-1)(2\ell-4)}{(\ell^2+\ell-2)(2\ell-4)+\ell(2\ell-1)}=\frac{2\ell^3-6\ell^2+2\ell +4}{2\ell^3-9\ell +8}n\le f(\ell)n,
\]
a contradiction.

In the following, suppose that $k\ge 2$. It follows from $k\le \ell-2$ that $\ell\ge 4$.

\textbf{Case 3}. $2 \leq k \leq \frac{\ell+3}{4}$.

In this case, we have  $\ell \ge 5$. By~\eqref{align2} and  Claim~\ref{Claim:X}, we have
\begin{align}
t^{+}(G) &\leq \frac{|X|(|E_1|- (3\ell -1))+(n-|X|)(|E_1|-(2\ell -k))}{|E_1|}\notag\\
&\leq
\frac{|X|(|E_1|- (3\ell -1))+(n-|X|)(|E_1|-(\frac{7}{4}\ell-\frac{3}{4}))}{|E_1|}\notag\\
&= n- \frac{7\ell-3}{4|E_1|}n-\frac{5\ell-1}{4|E_1|}|X|\notag\\
&= n- \frac{7\ell-3}{4|E_1|}n-\frac{5\ell-1}{4|E_1|}\left(n-\frac{k(n-t^{+}(G))}{2}\right), \notag
\end{align}
where the second inequality holds as $k \leq \frac{\ell+3}{4}$. This means
\[
\left(8|E_1|+(5\ell-1)k\right)t^{+}(G)\le \left(8|E_1|-2(7\ell-3)-2(5\ell-1)+(5\ell-1)k\right)n.
\]
Recall that $|E_1|=\ell^2 +\ell -\frac{1}{2}k^2-\frac{1}{2}k-1$. We have
\begin{align}\label{upper-bound-t-case3}
t^{+}(G)\leq \left(1- \frac{8(3\ell -1)}{-4k^2 +(5\ell - 5)k + 8(\ell^2+\ell -1)}\right)n.
\end{align}
Let
\[
g(k) =-4k^2 +(5\ell - 5)k.
\]
It follows from  $\ell\ge 5$ that $g(k)$ is monotone increasing  on the interval $[0, \frac{\ell+3}{4}]$. So,
\[
g(k) \leq g(\frac{\ell+3}{4}) = \ell^2+\ell -6.
\]
Substituting it into \eqref{upper-bound-t-case3} yields that
\begin{align}
t^{+}(G) &\leq \left(1- \frac{8(3\ell -1)}{\ell^2+\ell -6 + 8(\ell^2+\ell -1)}\right)n\notag\\
&= \left(1- \frac{8(3\ell -1)}{9\ell^2+9\ell -14}\right)n\notag\\
&<\left(1- \frac{6}{3\ell-1}\right)n,\notag
\end{align}
a contradiction.

\textbf{Case 4}. $k > \frac{\ell+3}{4}$.

By~\eqref{align3} and  Claim~\ref{Claim:X}, we have
\begin{align}
t^{+}(G)&\leq \frac{|X|(|E_2|-3k)+(n-|X|)(|E_2|-(2\ell-2k+1))}{|E_2|} \notag\\
&\leq \frac{|X|(|E_2|-3k)+(n-|X|)(|E_2|-(\ell-k+3))}{|E_2|} \notag\\
&= n- \frac{(\ell-k+3)}{|E_2|}n-\frac{(4k-\ell-3)}{|E_2|}|X|\notag\\
&\le  n- \frac{(\ell-k+3)}{|E_2|}n-\frac{(4k-\ell-3)}{|E_2|}\left(n-\frac{k(n-t^{+}(G))}{2}\right), \notag
\notag
\end{align}
where the second inequality holds as $k \leq \ell -2$. This means
\[
t^{+}(G)\le \left(1-\frac{6k}{2|E_2|+4k^2-k\ell-3k}\right)n,
\]
which together with $|E_2|=2\ell k-2k^2+k$ yields that
\[
t^{+}(G)\le  \frac{3\ell -7}{3\ell - 1}n.
\]

We can get a contradiction in all cases, which completes the proof.

\subsection{Tightness of Theorem~\ref{main result}}\label{sub4.2}
In this subsection, we show that the bound on positive co-degree in Theorem~\ref{main result} is best possible by constructing a $\mathcal{K}_{\ell+1}^3$-saturated $(\ell+1)$-partite $3$-graph $\mathcal{H}$ on $n$ vertices with $\delta_{2}^{+}(\mathcal{H})= f(\ell) n$. If $\ell=3$, we choose $n$ with $7\mid n$ and partition the vertex set
\begin{align}
V(\mathcal{H}) = A_1 \cup \cdots \cup A_5 \cup B_1\notag
\end{align}
such that
\begin{align}
|A_1|= \cdots = |A_5| =\frac{n}{7}~~\text{and}~~ |B_1|=\frac{2n}{7}.\notag
\end{align}
For the case that $\ell \geq 4$, we choose $n$ with $(3\ell-1)\mid n$  and partition the vertex set
\begin{align}
V = A_1 \cup \cdots \cup A_5 \cup B_1\cup \cdots \cup B_{\ell-2}\notag
\end{align}
such that
\begin{align}
|A_1|= \cdots = |A_5| =\frac{n}{3\ell-1}~~\text{and}~~ |B_1| = \cdots = |B_{\ell-2}| =\frac{3n}{3\ell-1}.\notag
\end{align}

The edge set of $\mathcal{H}$ consists of the following three types.
\begin{itemize}
\item all edges that intersect $A_i$, $A_{i+1}$ and $B_{j}$ on exact one vertex, where the indices of $A_i$, $A_{i+1}$ are taken modulo $5$ and $j \in [\ell -2]$;

\item all edges that intersect $A_i$, $B_j$ and $B_{j'}$ on exact one vertex, where $i \in [5]$ and $j,j' \in [\ell -2]$ are distinct;

\item all edges that intersect $B_j$, $B_{j'}$ and $B_{j''}$ on exact one vertex, where $j,j',j'' \in [\ell -2]$ are distinct.
\end{itemize}

It is easy to check that $\partial\mathcal{H}$ is a blowup of some $5$-wheel-like graph with   chromatic number  $\ell+1$, which implies $\mathcal{H}$ is  $(\ell+1)$-partite by Observation~\ref{l-partite}. The argument of Brandt~\cite{B03com} shows that $\partial\mathcal{H}$ is a $K_{\ell+1}$-saturated graph. By Lemma~\ref{Maximal Kl if and only if maximal Kr}, $\mathcal{H}$ is a $\mathcal{K}_{\ell+1}^3$-saturated 3-graph. Now we bound the co-degree $d(x,y)$ of a pair $x,y\in V(\mathcal{H})$. If $x\in B_j, y\in B_{j'}$ or $x \in B_j, y\in A_i$, where $i \in [5]$ and $j,j' \in [\ell -2]$ are distinct, then  $d(x,y)=\frac{3\ell-7}{3\ell-1}n$. If $x\in A_i, y \in A_{i+1}$, where the indices are taken modulo $5$, then $d(x,y)=\frac{3\ell-6}{3\ell-1}n$. For the other cases, $d(x,y)=0$. So, $\delta_{2}^{+}(\mathcal{H})=\frac{3\ell-7}{3\ell-1}n$ if $\ell\ge 4$. The case that $\ell=3$ is more simpler and we omit it here.

\section{Concluding Remarks}\label{Con}

In this paper, we give two stability results for $\mathcal{K}_{\ell + 1}^{r}$-saturated $r$-graphs. In fact, the stability of hypergraphs has been studied widely, one of which is the following degree-stability of $\mathcal{K}_{\ell+1}^r$ given by  Liu, Mubayi and Reiher~\cite{LMR21}.

\begin{theorem}[Liu, Mubayi and Reiher~\cite{LMR21}]\label{THM:liu degree}
For $\ell \geq r \geq 3$ there exist $\varepsilon > 0$ and $N_0 \in \mathbb{N}$ such that every $\mathcal{K}_{\ell+1}^r$-free $r$-graph $\mathcal{H}$ on $n \geq N_0$ vertices with $\delta(\mathcal{H}) \geq \left(\binom{\ell-1}{r-1}/\ell^{r-1}- \varepsilon \right)n^{r-1}$ is $\ell$-partite.
\end{theorem}

However, they did not give an explicit bound on $\varepsilon$. Now we give a bound on $\varepsilon$ by  Andr\'{a}sfai-Erd\H{o}s-S\'{o}s Theorem.

\begin{theorem}\label{end}
Let $\ell \geq r \geq 3$ be integers and $\varepsilon=\binom{\ell-1}{r-1}\left(\frac{1}{\ell^{r-1}}-\left(\frac{3\ell-4}{3\ell^2-4\ell+1}\right)^{r-1}\right)$. Then every $\mathcal{K}_{\ell+1}^r$-free $r$-graph $\mathcal{H}$ on $n$ vertices with $\delta(\mathcal{H}) > \left(\binom{\ell-1}{r-1}/\ell^{r-1}- \varepsilon \right)n^{r-1}$ is $\ell$-partite.
\end{theorem}
\begin{proof} Suppose that $\mathcal{H}$ is a $\mathcal{K}_{\ell+1}^r$-free $r$-graph $\mathcal{H}$ on $n$ vertices with
\[
\delta(\mathcal{H}) > \left(\binom{\ell-1}{r-1}/\ell^{r-1}- \varepsilon \right)n^{r-1}.
\]
 Let $G = \partial_{r-2}\mathcal{H}$. For every vertex $v$, $d_{\mathcal{H}}(v)$ is at most the number of $K_r$ which contains the vertex $v$ in $G$ by Observation~\ref{ob1}. Then we have $\delta(\mathcal{H}) \leq d_{\mathcal{H}}(v) \leq N(K_{r-1}, G[N_G(v)])$. It follows from $G$ is $K_{\ell+1}$-free that $G[N_G(v)]$ is $K_{\ell}$-free. Hence,
\begin{align}
\left(\binom{\ell-1}{r-1}\frac{1}{\ell^{r-1}}-\varepsilon\right)n^{r-1}< \delta(\mathcal{H})\leq N(K_{r-1}, G[N_G(v)]) \leq \binom{\ell-1}{r-1}\left(\frac{d_{G}(v)}{\ell-1}\right)^{r-1},\notag
\end{align}
which implies that $d_{G}(v) > \frac{3\ell-4}{3\ell-1}n$. Then  by  Andr\'{a}sfai-Erd\H{o}s-S\'{o}s Theorem, $G$ is $\ell$-partite. which means $\mathcal{H}$ is also $\ell$-partite by Observation~\ref{l-partite}.
\end{proof}
We believe that the bound in Theorem~\ref{end} is not best possible. It is natural to ask the following problem.

\begin{problem}
Let $\ell \geq r \geq 3$ be integers. Find the maximum $\varepsilon=\varepsilon(\ell,r)$ such that  every $\mathcal{K}_{\ell+1}^r$-free $r$-graph $\mathcal{H}$ on $n$ vertices with  $\delta(\mathcal{H}) > \left(\binom{\ell-1}{r-1}/\ell^{r-1}- \varepsilon \right)n^{r-1}$ is $\ell$-partite for large $n$.
\end{problem}

Theorem \ref{main result} gives a 3-graph version of  Andr\'{a}sfai-Erd\H{o}s-S\'{o}s Theorem and  is the first stability result on minimum positive co-degree for hypergraphs. Note that most stability results of a family $\mathcal{F}$ of hypergraphs are that every $\mathcal{F}$-free hypergraph $\mathcal{H}$ of size (or minimum degree) close to $\mathrm{ex}(n,\mathcal{F})$ (or minimum degree of its extremal construction)
can be transformed to its extremal construction by deleting and adding very few edges. The positive co-degree given in  Theorem \ref{main result} is much smaller than that in $\mathcal{K}_{\ell + 1}^{3}$'s extremal construction $T_3(n,\ell)$, and the extremal construction of  Theorem \ref{main result} is far from  $T_3(n,\ell)$ in edit-distance. As far as we know, this is  the first result with such an interesting property in hypergraphs.

The  proof of Theorem \ref{main result}  depends on a structural lemma of $K_{\ell+1}$-saturated graphs (Lemma \ref{Lemma:maximal}). It is easy to get $5$-wheel-like subgraph in the shadow of $\mathcal{K}_{\ell + 1}^{3}$-saturated $3$-graphs with large minimum positive co-degree. We do not know whether such a subgraph exists in the shadow of $\mathcal{K}_{\ell + 1}^{3}$-free $3$-graphs $\mathcal{H}$, since  adding a new edges to  $\mathcal{H}$ might reduce the minimum positive co-degree. So we pose the following problem.

\begin{problem}
Let $\ell \geq 3$ be an integer, and define
\[
f(\ell)= \begin{cases}
2/7 \quad & \ell=3,\\
\frac{3\ell-7}{3\ell-1} \quad & \ell \geq 4.
\end{cases}
\]
Is it true that every $\mathcal{K}_{\ell+1}^3$-free $3$-graph $\mathcal{H}$ on $n$ vertices with $\delta_{2}^{+}(\mathcal{H}) > f(\ell) n$ is $\ell$-partite?
\end{problem}

\bibliographystyle{abbrv}
\bibliography{A-stability-theorem-for-maximal-K-free-graphs}
\end{document}